\documentclass[12pt]{article}
\usepackage{amsmath}
\usepackage{amssymb}
\usepackage{amsthm}
\usepackage{yhmath}
\usepackage{mathrsfs}
\usepackage{graphicx}
\usepackage{epstopdf}
\usepackage{multirow}
\usepackage{longtable}
\usepackage{mathtools}
\usepackage{subfigure}
\usepackage{float}
\usepackage{geometry}
\usepackage{galois}
\usepackage{color}
\usepackage{diagbox}
\usepackage{empheq}
\usepackage{textcomp}
\usepackage[greek,english]{babel}
\usepackage{framed}
\usepackage{bm}
\usepackage{appendix}
\usepackage{cite}
\usepackage[numbers,compress]{natbib}
\usepackage{esvect}
\usepackage{tikz}
\usepackage{cases}

\makeatletter
\@addtoreset{equation}{section}
\makeatother
\definecolor{shadecolor}{gray}{0.85}
\newcommand{\arraystrech}[5]

\newtheorem{Lemma}{Lemma}
\newtheorem{Theorem}{Therorem}

\newtheorem{rem}{Remark}

\DeclareSymbolFont{largesymbol}{OMX}{yhex}{m}{n}
\DeclareMathAccent{\widehat}{\mathord}{largesymbol}{"62}

\begin{document}
\pagestyle{plain}

\title{A Locking-free DP-Q2-P1 MFEM for \\ Incompressible
Nonlinear Elasticity Problems\thanks{The research was supported by the NSFC
projects 11171008 and 11571022.}}

\author{Weijie Huang, \hspace{1mm} Zhiping Li\thanks{Corresponding author,
email: lizp@math.pku.edu.cn} \\ {\small LMAM \& School of Mathematical
Sciences, Peking University, Beijing 100871, China}}

\date{}

\maketitle
\begin{abstract}
A mixed finite element method (MFEM), using dual-parametric piecewise biquadratic
and affine (DP-Q2-P1) finite element approximations for the deformation and the pressure like Lagrange
multiplier respectively, is developed and analyzed for the numerical computation of
incompressible nonlinear elasticity problems with large deformation gradient,
and a damped Newton method is applied to solve the resulted discrete problem.
The method is proved to be locking free and stable.
The accuracy and efficiency of the method are illustrated by numerical experiments on
some typical cavitation problems.
\end{abstract}


\section{Introduction}

For incompressible elasticity, it is well known that, even in the case of small
deformation and linear elasticity, the notorious volume locking can happen and
ultimately leads to the failure of some finite element approximations
\cite{Brezzi,John1999,Rohan,Vivette1986}.
In the case of incompressible linear elasticity, it is well-known how to overcome locking
numerically, for example, by using the enhanced assumed strain methods to increase the
degrees of freedom of the elements \cite{SIMO1990, Chavan2007}, by using the nonconforming
finite element methods to weakening
the global continuity of the numerical solutions \cite{Kouhia1995}, or by
using the mixed finite element methods (MFEMs) to relax the constraint of the
incompressibility on the numerical solutions \cite{Brezzi,Ming}, etc.. However, for incompressible
nonlinear elasticity, especially for large deformation gradient problems which will be addressed
in this paper, there still lack of systematic results.

Let $\Omega\subset \mathbb{R}^2$ be a bounded open domain with smooth boundary occupied
by an isotropic hyper-elastic body in its reference configuration.
Let the stored energy density function of the material
$W(\nabla\bm{u}):M_{+}^{2\times 2}\to \mathrm{R}^+$ be poly-convex,
where $\bm{u}$ is a deformation field and
$M_+^{2 \times 2}$ is the set of $2 \times 2$ matrices with positive eigenvalues.
Since the material is incompressible, the deformation field must satisfy
the constraint $\det\nabla\bm{u} = 1\ a.e.$ in $\Omega$. In the mixed formulation of
nonlinear hyper-elasticity boundary value problems, one considers to solve the saddle point
problem
\begin{equation}
\label{mixed formulation}
(\bm{\tilde{u}},\tilde{p}) = \arg \inf_{\bm{u} \in \mathcal{A}} \sup_{p \in L^2(\Omega )}
E(\bm{u},p),
\end{equation}
where $p$ is the pressure like Lagrangian multiplier (see \cite{Brezzi}),
$E(\bm{u},p)$ is the Lagrangian functional defined as
\begin{equation}
  \label{functional}
  E(\bm{u}, p)=\int_{\Omega} \left( W(\nabla \bm{u}(\bm{x})) - p(\det\nabla u - 1)\right)\,
  \mathrm{d}\bm{x} - \int_{\partial_N \Omega} \bm{t}
  \cdot \bm{u} \,\mathrm{d}s,
\end{equation}
with $\bm{t}$ the traction imposed on the Neumann boundary $\partial_N\Omega$,
and where the set of admissible deformation functions $\mathcal{A}$ is given by
\begin{equation}
  \label{admissible set}
  \mathcal{A} =
  \begin{cases}
  \{\bm{u}\in W^{1,s}(\Omega;\mathbb{R}^2) \ \mbox{is 1-to-1
  a.e.}: \bm{u}|_{\partial_D\Omega} = \bm{u}_0,
  \}, & \text{if $\partial_D\Omega \neq \emptyset$,}\\
  \{\bm{u}\in W^{1,s}(\Omega;\mathbb{R}^2) \ \mbox{is 1-to-1 a.e.}:
  \int_{\Omega}\bm{u}\, \mathrm{d}\bm{x} = \bm{0},
  \}, & \text{otherwise,}
  \end{cases}
\end{equation}
where $s>1$ is a given Sobolev index, and $\partial_D\Omega$ is the Dirichlet
boundary with its 1-D measure $|\partial_D\Omega|\neq 0$.

The variational form of the Euler-Lagrange equation, {\em i.e.} the equilibrium equation,
of the mixed formulation \eqref{mixed formulation}, can be expressed as
\begin{equation}
\label{Euler Lagrange equation}
\left\{
\begin{aligned}
\int_{\Omega}\left( \dfrac{\partial W(\nabla \bm{u})}{\partial \nabla \bm{u}}:\nabla\bm{v}
-p\left(\operatorname{cof}\nabla \bm{u}:\nabla \bm{v}\right)\right) \,\mathrm{d}
\bm{x}&=\int_{\partial_N \Omega}\bm{t}\cdot\bm{v} \,\mathrm{d}s,\quad
\forall \bm{v}\in \mathcal{X}_0, \\
\int_{\Omega}q\left(\det\nabla \bm{u}-1\right) \,\mathrm{d}\bm{x}&=0,\quad \qquad \qquad \;\;\;
\forall q\in \mathcal{M},
\end{aligned}
\right.
\end{equation}
where $\operatorname{cof}\nabla\bm{u}$ denotes the cofactor matrix of $\nabla\bm{u}$, and
\begin{equation}
\mathcal{M} := L^2(\Omega), \;\;\;
\mathcal{X}_0 = \begin{cases}
\left\{\bm{v}\in H^1(\Omega;\mathbb{R}^2): \bm{v}|_{\partial_D\Omega} = \bm{0}\right\},
& \text{if $\partial_D\Omega \neq \emptyset$},\\
\left\{\bm{v}\in H^1(\Omega;\mathbb{R}^2): \int_{\Omega}\bm{v}\,\mathrm{d}\bm{x} =
\bm{0}\right\}, & \text{otherwise},\\
\end{cases}
\end{equation}
are the test function spaces for the pressure $p$ and deformation $\bm{u}$ respectively.

In the present paper, based on the variational form of Euler-Lagrange equation
\eqref{Euler Lagrange equation},  a mixed finite element method (MFEM),
using dual-parametric piecewise biquadratic and affine (DP-Q2-P1) finite element
approximations for the deformation $\bm{u}$ and pressure like Lagrangian
multiplier $p$ respectively, is developed and analyzed for the numerical computation of
incompressible nonlinear elasticity boundary value problems with large deformation gradient,
and a damped Newton method is applied to solve the resulted discrete problem.
The method is shown to be stable (locking free) under some reasonable assumptions
on the mesh regularity (see (M1) and (M2) in \S~2.1), the damping criteria (see (C1) and (C2) in
\S~2.2) and the stability hypothesis on the mixed formulation (see (H) in \S~2.3).
The performance of a DP-Q2-P1 method applied to a cavitation problem, which
shows an extremely large anisotropic deformation near the cavity surface,
is illustrated by numerical experiments and results.
We would like to point out here that the classical stability analysis for Q2-P1 element based on
the divergence free argument do not directly apply to nonlinear elasticity problems
with finite deformation (see for example \cite{Ming}), especially those with
very large nearly singular deformation gradients (see \S~2.3 for details).
The advantage of using the dual-parametric finite elements is that the elements can
well accommodate very large anisotropic deformation as well as
complex physical domain with a reasonable number of degrees of freedom
\cite{Lian Dual,Xu and Henao 2011,SuLiRectan}.

The rest of the paper is organized as follows. \S~2 is devoted to the construction of the
DP-Q2-P1 MFEM and its stability analysis. In \S~3,
the DP-Q2-P1 mixed finite element method is applied to a cavitation problem, and
the accuracy and efficiency of the method is demonstrated by the numerical results.
Some concluding remarks are given in \S~4.

\section{The mixed finite element method and its stability}

\subsection{The DP-Q2-P1 mixed finite element}

Let $(\hat{T}, \hat{P}, \hat{\Sigma})$ be the standard biquadratic-linear mixed
rectangular element with
\begin{align*}
\begin{cases}
\hat{T}=[-1, 1]\times [-1, 1], \\
\hat{P}=\{Q_2(\hat{T}), P_1(\hat{T})\},\\
\hat{\Sigma}=\{\bm{\hat{u}}(\hat{a}_i), 0\le i\le 8; \;\;
\hat{p}(\hat{b}_0), \partial_{\hat{x}_1} \hat{p}(\hat{b}_0),
\partial_{\hat{x}_2} \hat{p}(\hat{b}_0) \},
\end{cases}
\end{align*}
where $\{\hat{ a }_i\}_{i=0}^3$ are the vertices of $\hat{T}$, $\{\hat{ a }_i\}_{i=4}^7$
represent the nodes on the middle points of the corresponding edges of $\hat{T}$,
and $\hat{ a }_8=\hat{b}_0= (0,0)$, as shown in Figure~\ref{reference element}.

Given 4 non-degenerate vertices $\{a _i\}_{i=0}^3$ in anticlockwise order,
5 properly distributed vertices $\{a_i\}_{i=4}^8$, and a smooth injection $F_T$
satisfying $a_i = F_T(\hat{a}_i)$, $i=0,\cdots 8$, then $T = F_T(\hat{T})$ defines
a (curve edged) quadrilateral element (see for example Figure~\ref{general_quad}).
In applications, the most commonly used $F_T$ are bilinear, biquadratic and trigonometric
(see \eqref{F_T}).

\begin{figure}[htb]
 \begin{minipage}[l]{0.49\textwidth}\hspace*{1mm}
	\begin{minipage}[l]{0.225\textwidth}
    \centering    \includegraphics[width=1.4in]{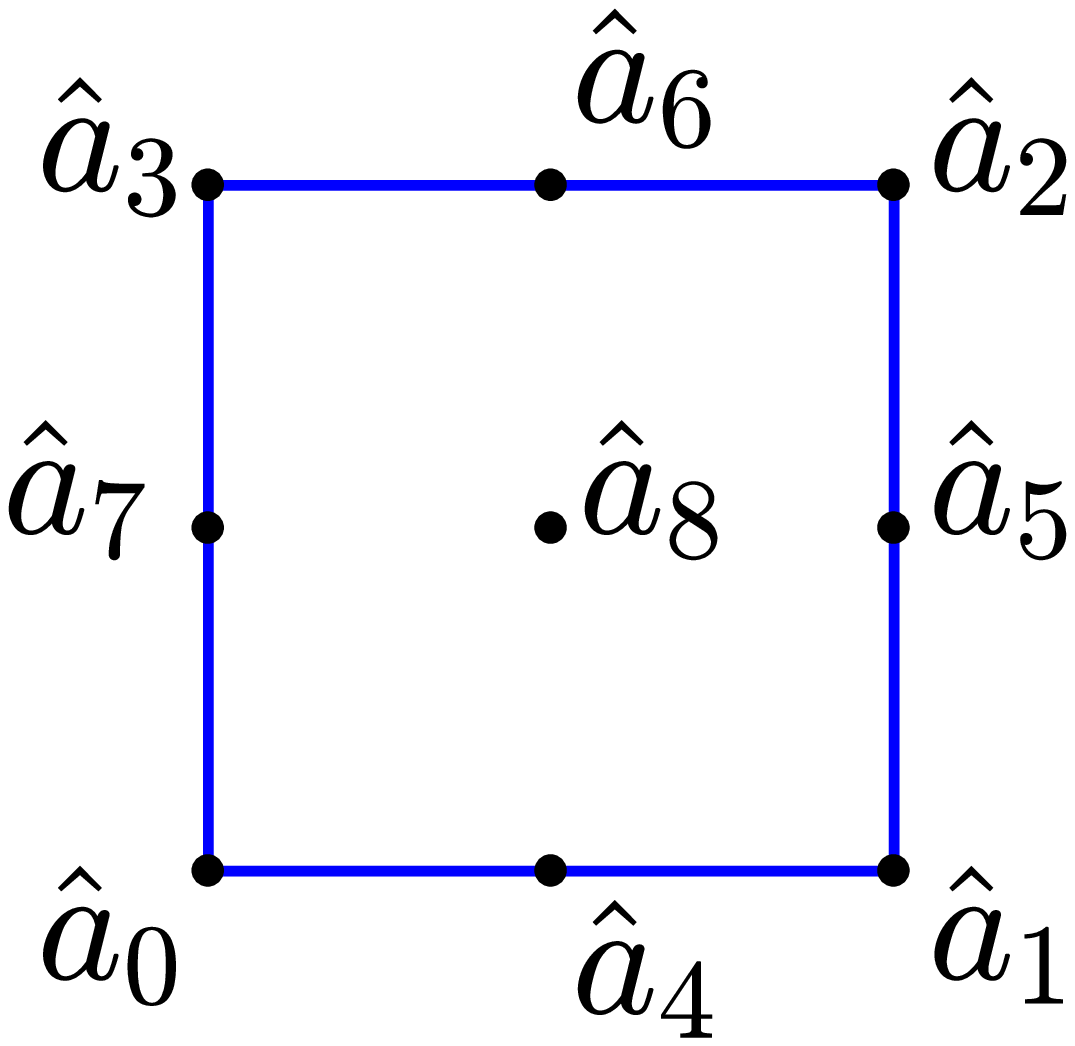}
	\vspace*{-5mm}
    \end{minipage}\hspace*{13mm}
    \begin{minipage}[r]{0.225\textwidth}
    \centering   \includegraphics[width=1.4in]{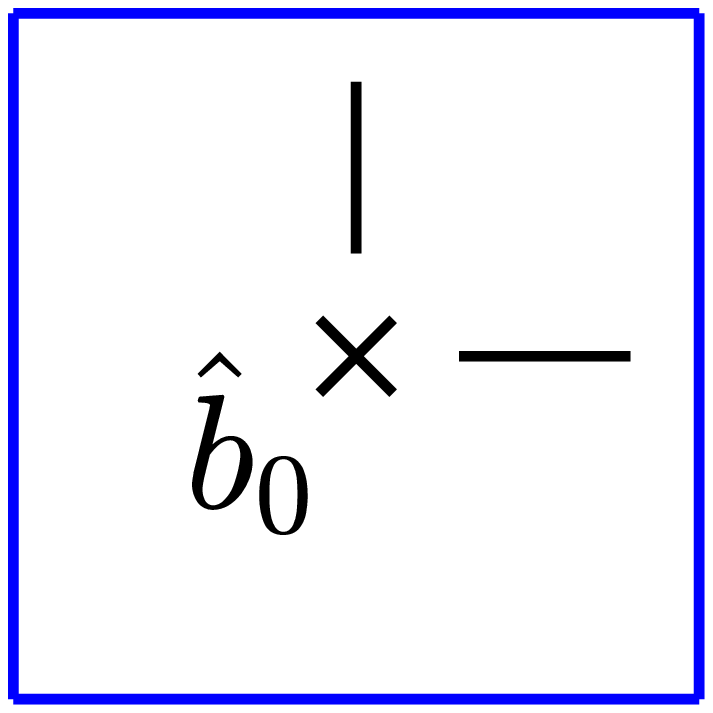}
    \vspace*{-5mm}
    \end{minipage}
    \caption{Reference element $\hat{T}$, $\hat{\Sigma}$.}\label{reference element}
 \end{minipage}
 \begin{minipage}[r]{0.49\textwidth}\hspace{5mm}
    \begin{minipage}[l]{0.225\textwidth}\vspace*{-3mm}
    \centering    \includegraphics[width=1.4in]{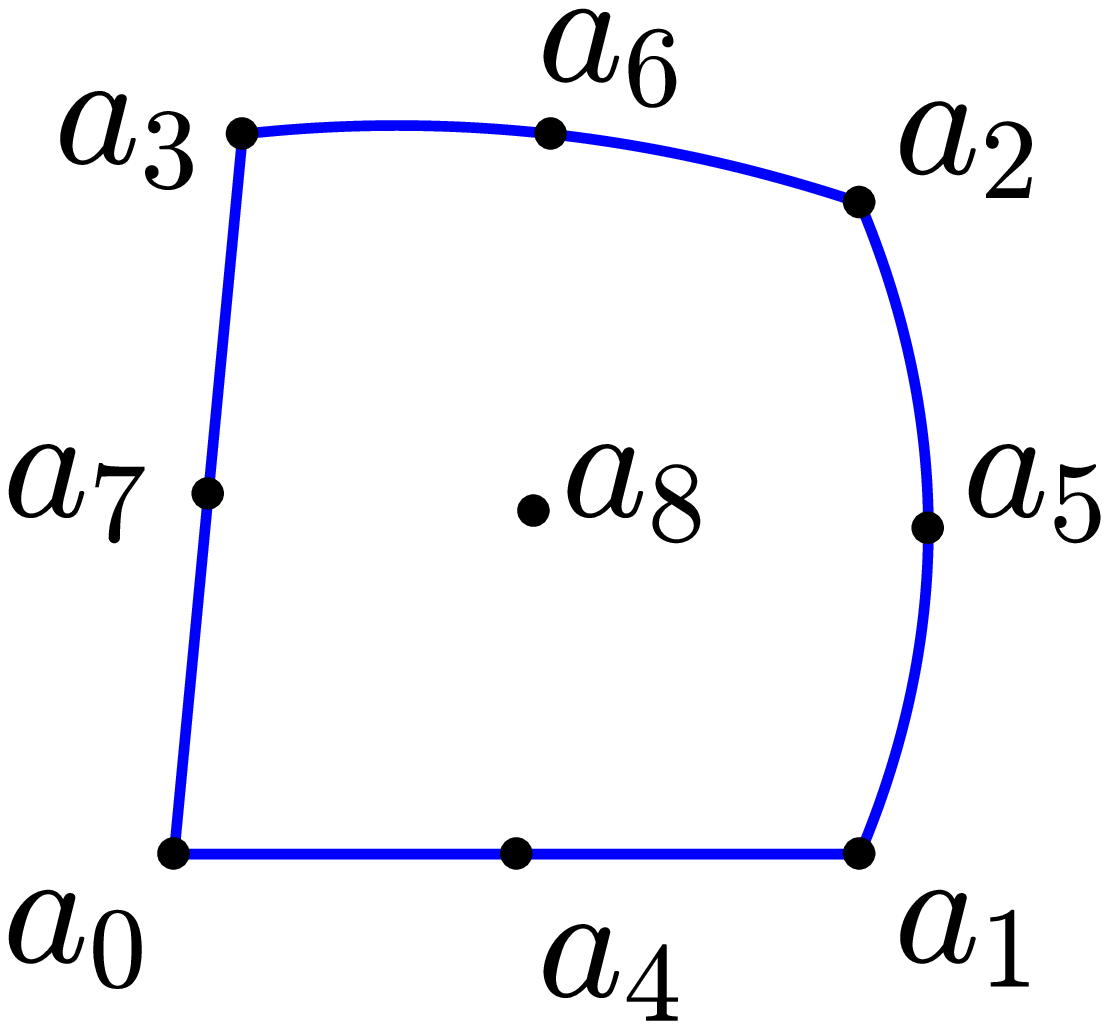}
	\vspace*{-5mm}
    \end{minipage}\hspace*{13mm}
    \begin{minipage}[r]{0.225\textwidth}
    \centering   \includegraphics[width=1.4in]{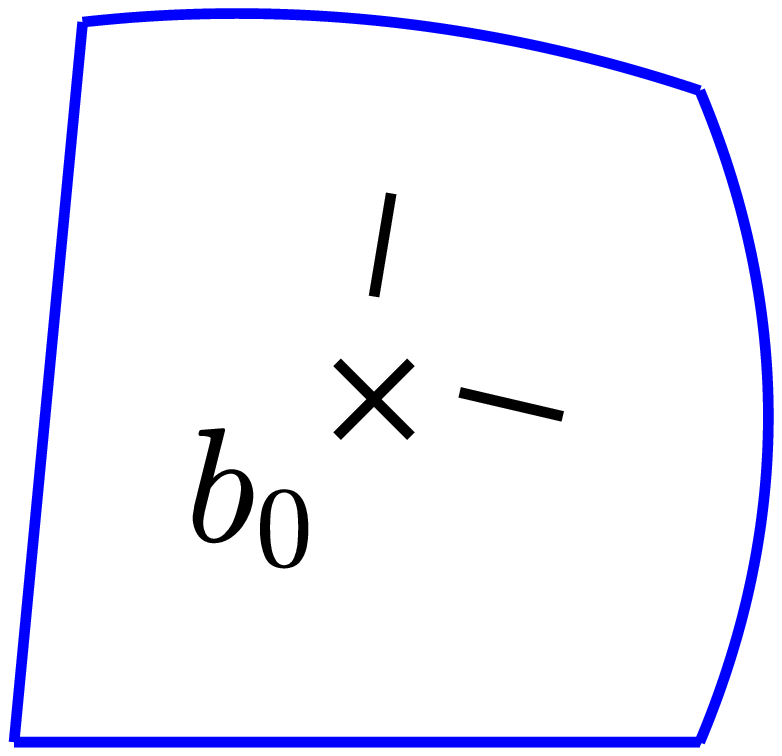}
    \vspace*{-5mm}
    \end{minipage}
    \caption{Element $T$, $\Sigma_T$.}\label{general_quad}
 \end{minipage}
\end{figure}

We define the dual-parametric biquadratic-affine (DP-Q2-P1) mixed finite element
$(T, P, \Sigma)$ as follows:
\begin{equation*}
\begin{cases}
T=F_T(\hat{T})\ \text{being a (curved) quadrilateral element,}\\
P_T=\big\{(\bm{u}, p) : T\to \mathbb{R}^2\times \mathbb{R}\ | \ \bm{u}=\hat{\bm{u}}\comp F_T^{-1},
\hat{\bm{u}}\in Q_2,\ p=\hat{p}\comp F_T^{-1}, \hat{p}\in P_1\big\},\\
\Sigma_T=\big\{\bm{u}(a_i),  a_i=F_T(\hat{ a }_i), 0\le i\le 8;\;
p(\bm{b}_0), \partial_{\hat{x}_1}p(\hat{b}_0)\comp F_T^{-1}, \partial_{\hat{x}_2} p(\hat{b}_0)\comp F_T^{-1},
b_0=F_T(\hat{b}_0) \big\},
\end{cases}
\end{equation*}
and denote
$\tilde{Q}_2 \times \tilde{P}_1 = (Q_2\comp F_T^{-1})\times (P_1\comp F_T^{-1})$.

For simplicity, we assume in this section that $\Omega = \Omega_h$ is properly
partitioned into such quadrilateral elements, {\em i.e.}
$\Omega = \Omega_h = \cup_{T\in \mathscr{T}_h} F_{T}(\hat{T})$.
In addition, we assume the triangulation $\mathscr{T}_h$ satisfies the following regularity
conditions.
\begin{description}
\item[{\bf (M1)}] The edge lengths are of quasi-uniform, {\em i.e.}
$|\wideparen{ a_0 a_3}|\cong |\wideparen{ a_0 a_1}|\cong |\wideparen{ a_1 a_2}|\cong |\wideparen{ a_2 a_3}|$, $h_T\cong h$, \
$\forall T\in \mathscr{T}_h$.

\item[\bf (M2)] The minimum angle condition, {\em i.e.}
$|\bm{l}_1|\cong |\bm{l}_2| \cong O(h_T)$ and $ h_T^{-2}|\bm{l}_1 \wedge \bm{l}_2|\ge c >0$,
$\forall T\in\mathscr{T}_h$, where $\bm{l}_1 = (\vv{a_0a_1} +\vv{a_3a_2} + 8\vv{a_7a_5})$
and $\bm{l}_2 = (\vv{a_0a_3} +\vv{a_1a_2} + 8\vv{a_4a_6})$.
\end{description}
Here and throughout the paper, $X \cong Y$, or equivalently $Y\lesssim X \lesssim Y$, means that
$c^{-1}Y \le X \le cY$ holds for a generic constant $c \ge 1$ independent of $T$ and $h$.

\begin{rem}
It is not difficult for us to show, by the standard scaling argument, that
\begin{equation}\label{scaling}
|\hat{\bm{v}}|_{\gamma,2,\hat{T}}\cong h_T^{\gamma-1}|\bm{v}|_{\gamma,2,T}, \;\; \gamma=0,1, \quad
\forall T \in \mathscr{T}_h, \; \text{and} \;\;\forall \bm{v}\in H^1(T;\mathbb{R}^2)
\end{equation}
remains valid, if $|\partial \bm{x}/\partial \bm{\hat{x}}| = |\partial F_T/\partial
\hat{\bm{x}}| \cong h_T$, and $\det (\partial \bm{x}/\partial \bm{\hat{x}})\cong h_T^2$,
which hold when the triangulation $\mathscr{T}_h$ satisfies (M1) and (M2).
\end{rem}

\subsection{The discretized problem}

Define the finite element function spaces for the admissible
deformation and pressure as
\begin{align}
\label{discrete admissible set A}
\mathcal{X}_h=\mathcal{A}_h :=&
\begin{cases}
\left\{\bm{v}_h\in C(\bar{\Omega};\mathbb{R}^2):\;
\bm{v}_h|_T\in \tilde{Q}_2, \ \bm{v}_h|_{\partial_D\Omega} = \bm{u}_0 \right\},
\quad\; \text{if $\partial_D\Omega\neq\emptyset$,}\\
\left\{\bm{v}_h\in C(\bar{\Omega};\mathbb{R}^2):\;
\bm{v}_h|_T\in \tilde{Q}_2, \ \int_{\Omega}\bm{v}_h\,\mathrm{d}\bm{x} =
\bm{0} \right\}, \quad \text{otherwise,}\\
\end{cases}\\
\mathcal{M}_h:= &\left\{p_h\in L^2(\bar{\Omega}):\;  p_h|_T\in \tilde{P}_1 \right\},
\label{discrete admissible set P}
\end{align}
and define the finite element test function space for the deformation as
\begin{align}
\mathcal{X}_{h,0}:=&
\begin{cases}
\left\{\bm{v}_h\in C(\bar{\Omega};\mathbb{R}^2):\;
\bm{v}_h|_T\in \tilde{Q}_2, \ \bm{v}_h|_{\partial_D\Omega} = \bm{0} \right\},
\quad\; \text{if $\partial_D\Omega\neq\emptyset$,}\\
\left\{\bm{v}_h\in C(\bar{\Omega};\mathbb{R}^2):\;
\bm{v}_h|_T\in \tilde{Q}_2, \ \int_{\Omega}\bm{v}_h\,\mathrm{d}\bm{x} =
\bm{0} \right\}, \quad \text{otherwise.}\\
\end{cases}
\end{align}

In the DP-Q2-P1 mixed finite
element method, the equilibrium equation \eqref{Euler Lagrange equation} is discretized into
the following form
\begin{equation}
\label{discrete Euler Lagrange equation}
\left\{
\begin{aligned}
\int_{\Omega }\dfrac{\partial W(\nabla\bm{u}_h)}{\partial\nabla\bm{u}_h}:\nabla\bm{v}_h-p_h
\operatorname{cof} \nabla \bm{u}_h:\nabla \bm{v}_h\,
\mathrm{d}\bm{x}&=\int_{\partial\Omega_N}\bm{t}\cdot \bm{v}_h \,\mathrm{d}s,
\quad \forall \bm{v}_h\in \mathcal{X}_{h,0}, \\
\int_{\Omega }q_h(\det\nabla \bm{u}_h-1) \,\mathrm{d}\bm{x}&=0,
\qquad \qquad \qquad \;  \forall q_h\in \mathcal{M}_h,
\end{aligned}
\right.
\end{equation}
and, in each iteration step of the damped Newton method to solve this discrete nonlinear problem,
one solves the following discrete linear problem
\begin{equation}
\label{Discrete Newton Method}
\left\{
\begin{aligned}
\text{Find } (\bm{w}_h,p_h) \in \mathcal{X}_{h,0} \times \mathcal{M}_h, \,\text{such} &
\;\text{that} \\ a(\bm{w}_h,\bm{v}_h; \underline{\bm{u}}_h,\underline{p}_h)+
b(\bm{v}_h,p_h;\bm{\underline{u}}_h) & = f(\bm{v}_h; \underline{\bm{u}}_h,\underline{p}_h),\quad
\forall \bm{v}_h\in \mathcal{X}_{h,0}, \\
b(\bm{w}_h,q_h;\bm{\underline{u}}_h)&=g(q_h;\bm{\underline{u}}_h), \qquad  \; \;\;
\forall q_h\in \mathcal{M}_h,
\end{aligned}
\right.
\end{equation}
to obtain a direction modifying $(\bm{w}_h,p_h)$, where $\underline{\bm{u}}_h
:= \bm{u}^k_h\in\mathcal{A}_h$, $\underline{p}_h := p^k_h$ denotes the approximation
solution obtained in the k-th iteration, and
\begin{eqnarray}
&&\label{a(;)} a(\bm{w}_h, \bm{v}_h; \underline{\bm{u}}_h,\underline{p}_h):=
\int_{\Omega} \left(\dfrac{\partial^2 W(\nabla\bm{u}_h)}{\partial(\nabla\bm{u}_h)^2}
\Big|_{\bm{u}_h = \underline{\bm{u}}_h} :\nabla\bm{w}_h\right):\nabla\bm{v}_h -
\underline{p}_h\operatorname{cof}\nabla\bm{w}_h:\nabla\bm{v}_h\,\mathrm{d}\bm{x},\\
&&\label{b(;)} b(\bm{v}_h, q_h;\underline{\bm{u}}_h) := \int_{\Omega}q_h
\operatorname{cof}\nabla\underline{\bm{u}}_h:\nabla\bm{v}_h\,\mathrm{d}\bm{x},\\
&&\label{f(;)} f(\bm{v}_h;\underline{\bm{u}}_h,\underline{p}_h) :=
\int_{\partial_N\Omega}\bm{t}\cdot \bm{v}_h\,\mathrm{d}s -
\int_{\Omega} \dfrac{\partial W(\nabla\bm{u}_h)}{\partial (\nabla\bm{u}_h)}
\Big|_{\bm{u}_h = \underline{\bm{u}}_h}:\nabla\bm{v}_h - \underline{p}_h
\operatorname{cof}\nabla\underline{\bm{u}}_h:\nabla\bm{v}_h\,\mathrm{d}\bm{x},\qquad \\
&&\label{g(;)} g(q_h;\underline{\bm{u}}_h) := -\int_{\Omega}q_h
(\det\nabla\underline{\bm{u}}_h - 1)\,\mathrm{d}\bm{x},
\end{eqnarray}
where $\frac{\partial^2 W(\nabla\bm{u}_h)}{\partial(\nabla\bm{u}_h)^2}$ is a fourth order
tensor. To simplify the notation, $a(\bm{w}_h,\bm{v}_h; \underline{\bm{u}}_h,\underline{p}_h)$,
$b(\bm{w}_h, q_h; \bm{\underline{u}}_h)$, $f(\bm{v}_h; \underline{\bm{u}}_h,\underline{p}_h)$ and
$g(q_h; \underline{\bm{u}}_h)$
will be denoted as $a(\bm{w}_h,\bm{v}_h)$, $b(\bm{w}_h,q_h)$, $f(\bm{v}_h)$, $g(q_h)$
whenever $(\underline{\bm{u}}_h, \underline{p}_h)$ are not directly involved in the calculation.

The whole solution process is summarized
as the following  algorithm.

\vspace*{3mm}
{\bf Algorithm}:
\begin{itemize}
\item Step 1. Provide the initial guess $(\bm{u}_h^0, p_h^0) \in \mathcal{X}_h
  \times \mathcal{M}_h$, the initial damping parameter
  $\alpha_0\in(0, 1]$, the tolerances $TOL$, $TOL'>0$, and set $k := 0$, $\alpha:=\alpha_0$.

\item Step 2. Set $\bm{\underline{u}}_h := \bm{u}_h^k$, and $\underline{p}_h := p_h^k$, and
  solve \eqref{Discrete Newton Method} to obtain
  $(\bm{w}_h,p_h)\in \mathcal{X}_{h,0}\times \mathcal{M}_h$.

\item Step 3. Set $\bm{u}_h^{k+1} := \bm{\underline{u}}_h + \alpha\bm{w}_h$ and
   $p_h^{k+1} := \underline{p}_h + \alpha p_h$.

\item Step 4. If $\bm{u}_h^{k+1}$ satisfies the criteria (C1)-(C2) given below,
go forward to Step 5; otherwise, set $\alpha := \alpha/2$, and go back to Step 3.

\item Step 5. If $\|\bm{u}_h^{k+1} - \bm{u}_h^k\|\le TOL$ and $|p_h^{k+1} - p_h^k| < TOL'$,
  stop; otherwise, set $k := k+1$, $\alpha:=\min\{\alpha_0, 2\alpha \}$ and go back to Step 2.
\end{itemize}

The following conditions are introduced as a criterion in the step~4 of the algorithm
to confine the iteration trajectory to
well behaved deformations, {\em i.e.} orientation preserving finite deformations
without too much oscillations on the deformation gradient.
\begin{description}
\item[{\bf (C1)}] $\sigma \le \lambda_1(\nabla\underline{\bm{u}}_h) \le
\lambda_2(\nabla\underline{\bm{u}}_h) \le 1/\sigma$, and
$0 < c \le \det\nabla\underline{\bm{u}}_h \le C$, $\forall \bm{x} \in\Omega$, where
$\lambda_1(\nabla\underline{\bm{u}}_h) \le \lambda_2(\nabla\underline{\bm{u}}_h)$ are the
eigenvalues of $\nabla\underline{\bm{u}}_h$, and $\sigma \in (0, 1)$,
$0<c < 1 < C$ are constants independent of $h$.

\item[{\bf (C2)}] $h_T |\underline{\bm{u}}_h|_{2,\infty,T} \le \underline{C}$,
   $\forall T\in \mathscr{T}_h$, where $\underline{C}>1$ is a given constant independent of $h$.
\end{description}

\begin{rem}
Notice that $\lambda_1(\nabla\bm{u})$ and $\lambda_2(\nabla\bm{u})$ are the principal strains
of the deformation $\bm{u}$, we see that (C1) is violated only if $\bm{u}$ is in a neighbourhood
of a singular deformation. In general, let $\bm{u}$ be a non-singular
solution to the problem, then it is necessary to choose
$\sigma < \inf_{\bm{x} \in \Omega} \min\{\lambda_1(\nabla \bm{u}),\lambda_2^{-1}(\nabla \bm{u})\}$.
In applications, (C2) can be easily satisfied unless the problem admits only microstructure
solutions, which consists of increasingly oscillatory energy minimizing sequences
\cite{Ball:James1}.
\end{rem}

Our numerical experiments on cavitation problems show that the damped Newton method applied here
in the above algorithm is as expected much more efficient than the modified Picard iteration used by
Lian and Li in \cite{Lian Dual}.

\subsection{Stability of the DP-Q2-P1 mixed finite element method}

To show the stability of the iso-parametric mixed finite element method for the
discrete linear problem \eqref{Discrete Newton Method}, we assume that:
\begin{description}
\item[{\bf (H)}] For $\underline{\bm{u}}\in \mathcal{A}\cap H^1(\Omega;\mathbb{R}^2)$
satisfying $(C1)$, $b(\bm{v},q;\underline{\bm{u}})$ satisfies inf-sup condition,
{\em i.e.} there exists a constant $\beta>0$ such that
\begin{equation}
\label{inf-sup condition}
\sup_{\bm{v}\in \mathcal{X}_0}\dfrac{b(\bm{v},q)}{\|\bm{v}\|_{1,2,\Omega}}\ge
\beta \|q\|_{0,2,\Omega},\quad \forall q\in L^2(\Omega).
\end{equation}
\end{description}

\begin{rem}
If in addition to (C1), $\underline{\bm{u}}$ satisfies certain regularity condition,
then $b(\bm{v},q)$ can be proved to satisfy the inf-sup condition
\eqref{inf-sup condition}(see\cite{Dobrowolski}).
\end{rem}

The key for the DP-Q2-P1 mixed finite element method to be stable and locking free for the problem
\eqref{Discrete Newton Method} is that the discrete inf-sup condition (or LBB condition)
\begin{equation}
\label{LBB condition}
\sup_{\bm{v}_h\in \mathcal{X}_{h,0}}\dfrac{b(\bm{v}_h,q_h;\bm{\underline{u}}_h)}{\Vert
\bm{v}_h\Vert _{1,2,\Omega }}\ge \beta\Vert q_h\Vert _{0,2,\Omega },
\quad  \forall q_h\in \mathcal{M}_h
\end{equation}
holds for a constant $\beta$ independent of the mesh size $h$, which can be established
by means of the famous Fortin criterion {(see Proposition 2.8 on page 58
in \cite{Brezzi})} and a general two steps construction frame as given in
Lemma~\ref{2 steps construction} (see Proposition 2.9 on page 59 in \cite{Brezzi}),
under the mesh regularity conditions (M1)-(M2), the deformation regularity conditions
(C1)-(C2) and the hypothesis (H).
Notice that, for nonlinear elasticity problems \cite{Dobrowolski, Ming},
$$
b(\bm{v}_h, q_h; \bm{\underline{u}}_h) =
\int_{\Omega_\rho}\operatorname{cof}\nabla\bm{\underline{u}}_h:\nabla\bm{v}_hq_h
\,\mathrm{d}\bm{x} = \int_{\Omega_\rho}\operatorname{div}((\operatorname{cof}
\nabla\bm{\underline{u}}_h)^T\bm{v}_h)q_h\,\mathrm{d}\bm{x},
$$
and for nearly singular deformation gradients problems, such as the cavitation problem,
$\operatorname{cof}\nabla\bm{\underline{u}}_h$
can be very ill conditioned. The following stability analysis reveals how the stability
constant $\beta$ depends on the condition number of $\operatorname{cof}\nabla\bm{\underline{u}}_h$,
and ultimately provides an inside perspective to the conditions (M1)-(M2) and (C1)-(C2),
which are crucial to the mesh generation and the iteration
(see step 4 of the algorithm).

Without loss of generality, in this subsection, we limit ourselves to the case
$\partial_D\Omega = \emptyset$. The theory for the case $\partial_D\Omega \neq\emptyset$
can be established in a similar way.

\begin{Lemma}
\label{Fortin Criterion}
(see Fortin Criterion \cite{Brezzi}) Let $b(\bm{v},q;\underline{\bm{u}}_h)$ satisfy the inf-sup condition \eqref{inf-sup condition}. The LBB condition \eqref{LBB condition} holds with a
constant $\beta$ independent of $h$ if and only if there exists an operator
$\Pi_h\in \mathscr{L}(\mathcal{X}_0,\mathcal{X}_{h,0})$ satisfying:
\begin{equation}
\label{the operator}
\left\{
\begin{aligned}
 &b(\bm{v}-\Pi_h\bm{v},q_h;\underline{\bm{u}}_h)=0,\qquad  \forall q_h\in \mathcal{M}_h,\
 \forall \bm{v}\in \mathcal{X}_0, \\
 &\Vert \Pi_h\bm{v}\Vert _{1,2,\Omega }\le c\Vert \bm{v}
 \Vert _{1,2,\Omega },\quad \quad \forall \bm{v}\in \mathcal{X}_0,
\end{aligned}
\right.
\end{equation}
with a constant $c>0$ independent of $h$.
 \end{Lemma}

\begin{Lemma}
\label{2 steps construction}
Let $\Pi_1\in \mathscr{L}(\mathcal{X}_0,\mathcal{X}_{h,0})$ and $\Pi_2\in\mathscr{L}(\mathcal{X}_0,\mathcal{X}_{h,0})$
be such that
\begin{subnumcases}{\label{operator construction}}
\Vert \Pi_1\bm{v}\Vert _{1,2,\Omega }\le c_1\Vert \bm{v}\Vert _{1,2,\Omega },
\qquad\qquad\;\;\; \forall \bm{v}\in \mathcal{X}_0,\label{oc:a} \\
\Vert \Pi_2(I-\Pi_1)\bm{v}\Vert _{1,2,\Omega }\le c_2\Vert
\bm{v}\Vert _{1,2,\Omega }, \quad\; \forall \bm{v}\in \mathcal{X}_0,\label{oc:b}\\
b(\bm{v}-\Pi_2\bm{v},q_h;\bm{\underline{u}}_h)=0, \quad  \forall \bm{v}\in \mathcal{X}_0,\
\forall q_h\in \mathcal{M}_h.\label{oc:c}
\end{subnumcases}
Set $\Pi_h\bm{v}=\Pi_1\bm{v}+\Pi_2(\bm{v}-\Pi_1\bm{v})$, then $\Pi_h\in
\mathscr{L}(\mathcal{X}_0,\mathcal{X}_{h,0})$ satisfies \eqref{the operator}.
\end{Lemma}

\emph{\textbf{Step 1}}. The construction of $\Pi_1\in \mathscr{L}(\mathcal{X}_0,\mathcal{X}_{h,0})$.
Let $(\bar{\mathcal{X}}_{h,0},\bar{\mathcal{M}}_h)$ be given by
\begin{equation}
\label{sub space}
\begin{cases}
\bar{\mathcal{X}}_{h,0} = \left\{\bm{v}_h \in \mathcal{X}_{h,0}:\; \bm{v}_h|_T \in
\tilde{Q}_1\oplus\text{span}\{\bm{p}_1,\bm{p}_2,\bm{p}_3,\bm{p}_4\} \right\} \\
\bar{\mathcal{M}}_h=\left\{q_h \in \mathcal{M}_h:\;q_h|_T \in \tilde{P}_0\right\},
\end{cases}
\end{equation}
where $\tilde{Q}_1= Q_1\comp F_T^{-1}, \tilde{P}_0=P_0\comp F_T^{-1}$, and $\{\bm{p}_{i}\}_{i=1}^{4}$
are the edge bubble functions with respect to the edges $\{e_{i}\}_{i=1}^{4}$ of $T$. For example,
for $i=1$, let $\bm{x}=(x_1,x_2)$ and
$\hat{\bm{x}}=F_T^{-1}(\bm{x})=(\hat{x}_1,\hat{x}_2)$, then
$$
\bm{p}_1(\bm{x})=(\hat{q}_1\comp F_T^{-1}(\hat{\bm{x}}))\bm{n}_1(F_T(-1,\hat{x}_2)),
$$
where $\hat{q}_1=(1-\hat{x}_2^2)(1-\hat{x}_1)$ and $\bm{n}_1$ is the unit out normal of
the edge $e_1$. The formulae for $\{\bm{p}_{i}\}_{i=2}^{4}$ are similar.
Obviously $\bm{p}_i(\bm{x}) =0$, $\forall \bm{x} \in \partial T \setminus e_i$.
In particular,  we notice that
$\{\bm{p}_{i}\}_{i=1}^{4}$
have zero tangential components on the edges of $T=F_T(\hat{T})$.

Firstly, let $\tilde{\Pi}_1: \mathcal{X}_0 \rightarrow
\left\{\bm{v}_h \in C(\Omega ; \mathbb{R}^2):\; \bm{v}_h|_T \in \tilde{Q}_1, \forall T \in
\mathscr{T}_h\right\}$ be the Cl\'{e}ment interpolation operator, since (M1) is satisfied, it follows
from the standard scaling argument (see for example Corollary 2.1 on page 106 in \cite{Brezzi}) that
\begin{equation}\label{Clement}
\sum_{T \in \mathscr{T}_h}h_T^{2\gamma-2}|\bm{v}-\tilde{\Pi}_1\bm{v}|_{\gamma,2,T}^2\lesssim
|\bm{v}|_{1,2,\Omega }^2,\quad \gamma = 0,1.
\end{equation}
Define $\bar{\Pi}_1 \bm{v} = \tilde{\Pi}_1 \bm{v} - \frac{1}{|\Omega |}
\int_{\Omega }\tilde{\Pi}_1 \bm{v} \,\mathrm{d}\bm{x}$, then
$\bar{\Pi}_1\in \mathscr{L}(\mathcal{X}_0,\mathcal{\bar{X}}_{h,0})$. Since,
$\int_{\Omega } \bm{v} \,\mathrm{d}\bm{x}=0$, it follows from the H\"{o}lder
inequality, \eqref{Clement} and  $h_T\cong h$ (see (M1)) that
\begin{equation*}
\begin{aligned}
\bigg|\dfrac{1}{|\Omega |}\int_{\Omega }\tilde{\Pi}_1\bm{v}\,\mathrm{d}\bm{x}
\bigg| \le \dfrac{1}{|\Omega |}\int_{\Omega }|\tilde{\Pi}_1\bm{v}-\bm{v}|
\,\mathrm{d}\bm{x}\lesssim \|\tilde{\Pi}_1\bm{v}-\bm{v}\|_{0,2,\Omega }\lesssim
h|\bm{v}|_{1,2,\Omega }.
\end{aligned}
\end{equation*}
Consequently, by \eqref{Clement} and $h_T\cong h$, we have
\begin{equation}\label{Clement2}
|\bm{v}-\bar{\Pi}_1\bm{v}|_{\gamma,2,\Omega }\lesssim
h^{1-\gamma}|\bm{v}|_{1,2,\Omega },\quad \gamma = 0,1.
\end{equation}
Next, let $\tilde{\Pi}_2: \mathcal{X}_0 \rightarrow
\{ \bm{v}_h \in C(\Omega ; \mathbb{R}^2) : \ \bm{v}_h|_T \in
\text{span}\{\bm{p}_1,\bm{p}_2,\bm{p}_3,\bm{p}_4\} \}$ be defined by
\begin{equation}
\label{operator2}
\left\{
\begin{aligned}
&\tilde{\Pi}_2\bm{v}|_T\in \text{span}\{\bm{p}_1,\bm{p}_2,\bm{p}_3,\bm{p}_4\},\\
&\int_{e_i}(\operatorname{cof}\nabla \underline{\bm{u}}_h^{\rm T}
\tilde{\Pi}_2\bm{v})\cdot \bm{n}_i\,\mathrm{d}s =
\int_{e_i}(\operatorname{cof}\nabla\underline{\bm{u}}_h^{\rm T}\bm{v})\cdot \bm{n}_i\,\mathrm{d}s,
\ \forall e_i = F_T^{-1} (\hat{e}_i), \ i=1,2,3,4.
\end{aligned}
\right.
\end{equation}
Define $\bar{\Pi}_2 \bm{v} = \tilde{\Pi}_2 \bm{v} - \frac{1}{|\Omega |}
\int_{\Omega }\tilde{\Pi}_2 \bm{v} \,\mathrm{d}\bm{x}$, then we have
$\nabla \bar{\Pi}_2\bm{v} \equiv \nabla\tilde{\Pi}_2\bm{v}$,
for all $\bm{v} \in H^1(\Omega ;\mathbb{R}^2)$, $\bar{\Pi}_2\in
\mathscr{L}(\mathcal{X}_0,\mathcal{\bar{X}}_{h,0})$, and in particular, as
$\operatorname{div}(\operatorname{cof}\nabla\underline{\bm{u}}_h|_T)=0$, we have
\begin{equation}
\label{third condition}
\begin{aligned}
\int_T\operatorname{cof}\nabla\underline{\bm{u}}_h:\nabla(\bar{\Pi}_2\bm{v}-\bm{v})\,\mathrm{d}\bm{x}
=\int_{\partial T}(\operatorname{cof}\nabla\underline{\bm{u}}_h^{\rm T}
(\tilde{\Pi}_2\bm{v}-\bm{v}))\cdot\bm{n}\,\mathrm{d}s=0.
\end{aligned}
\end{equation}

Now, define $\Pi_1 \in \mathscr{L}(\mathcal{X}_0,\mathcal{X}_{h,0})$ by setting
$\Pi_1 \bm{v} \triangleq \bar{\Pi}_h\bm{v}=\bar{\Pi}_1\bm{v}+\bar{\Pi}_2
(\bm{v}-\bar{\Pi}_1\bm{v})$, $\forall \bm{v} \in \mathcal{X}_0$.

\vskip 3mm
\emph{\textbf{Step 2.}} The construction of $\Pi_2 \in \mathscr{L}(\mathcal{X}_0,\mathcal{X}_{h,0})$.
Denote the bi-quadratic bubble function space on $\hat{T}$ by
$\hat{\bm{B}} = \{ \hat{\bm{b}}(\hat{\bm{x}})=(b_1(1-\hat{x}_1^2)(1-\hat{x}_2^2),
b_2(1-\hat{x}_1^2)(1-\hat{x}_2^2)) \}$. Define
\begin{equation}\label{bubble}
\mathcal{B}_h=\{ \bm{b}\in C(\bar{\Omega};\mathbb{R}^2) :
\bm{b}|_T=\hat{\bm{b}} \comp F_T^{-1}, \ \hat{\bm{b}}\in \hat{\bm{B}} \}.
\end{equation}
Notice that $\mathcal{X}_{h,0} = \bar{\mathcal{X}}_{h,0} + \mathcal{B}_h$. Define
$\Pi_2 : \{\bm{v}\in \mathcal{X}_0: \int_T\operatorname{cof}\nabla\underline{\bm{u}}_h:
\nabla\bm{v}\,\mathrm{d}\bm{x}=0\}\rightarrow \mathcal{B}_h$ as the unique solution of
\begin{equation}
\label{operator2 construction}
\int_T\operatorname{cof}\nabla\underline{\bm{u}}_h:\nabla(\Pi_2\bm{v}-\bm{v})\ q_h
\,\mathrm{d}\bm{x}=0,\ \forall q_h\in \tilde{P}_1(T) \setminus \tilde{P}_0(T), \
\forall T \in \mathscr{T}_h,
\end{equation}
Since $\Pi_2\bm{v}$ is a bubble function on $T \in \mathscr{T}_h$ and
$\operatorname{div}(\operatorname{cof}\nabla\bm{\underline{u}}_h|_T)=0$, we have
\begin{equation}\label{Pi v null}
\int_{T}\operatorname{cof}\nabla\underline{\bm{u}}_h:\nabla\Pi_2\bm{v} \,
\mathrm{d}\bm{x}=\int_{\partial T}(\operatorname{cof}\nabla
\underline{\bm{u}}_h)\bm{n}\cdot \Pi_2\bm{v}
\,\mathrm{d}s-\int_{T}\operatorname{div}(\operatorname{cof}\nabla\bm{\underline{u}}_h)
\cdot \Pi_2\bm{v}\, \mathrm{d}\bm{x} = 0.
\end{equation}

\begin{Lemma}\label{Pi_1 lemma}
Let $\mathscr{T}_h$ satisfy the condition (M1), and $\underline{\bm{u}}_h$ satisfy the condition (C1).
Then, $\Pi_1\in\mathscr{L}(\mathcal{X}_0,\bar{\mathcal{X}}_{h,0})$, as defined in step 1, satisfies
\begin{equation}
\left\{
\begin{aligned}
\label{operator1 property}
&\Vert \Pi_1\bm{v}\Vert _{1,2,\Omega }\lesssim\dfrac{1}{\sigma^2}\Vert \bm{v}
\Vert _{1,2,\Omega },\quad \forall \bm{v}\in \mathcal{X}_0,\\
&\int_T\operatorname{cof}\nabla \underline{\bm{u}}_h:
\nabla(\Pi_1\bm{v}-\bm{v})\,\mathrm{d}\bm{x}=0,\quad \forall \bm{v}\in \mathcal{X}_0,
\ \forall T\in \mathscr{T}_h.
\end{aligned}
\right.
\end{equation}
\end{Lemma}

\begin{proof}
Since $\tilde{\Pi}_2\in \mathscr{L}(\mathcal{X}_0,\mathcal{\bar{X}}_{h,0})$ is defined by
\eqref{operator2}, thus, by solving the linear system, $\tilde{\Pi}_2\bm{v}$ can be
explicitly expressed as $\tilde{\Pi}_2 \bm{v}=\sum\limits_{i=1}^{4}\alpha_i(\bm{v}) \bm{p}_i$,
where
\begin{equation}
\label{alpha i  definition}
\begin{aligned}
\alpha_i=\Big[\int_{e_i}(\operatorname{cof}\nabla\underline{\bm{u}}_h^T\bm{v})
\cdot\bm{n}_i\,\mathrm{d}s\Big]\Big/\Big[\int_{e_i}(\operatorname{cof}
\nabla\underline{\bm{u}}_h^T\bm{p}_i)\cdot \bm{n}_i\, \mathrm{d}s\Big],\quad i =1,2,3,4.
\end{aligned}
\end{equation}
Noticing that (C1) implies that $\lambda_2(\operatorname{cof}\nabla \underline{\bm{u}}_h)
\lesssim \sigma^{-1}$, by the trace theorem, we have
\begin{equation}\label{numerator}
\Big|\int_{e_i}(\operatorname{cof}\nabla\underline{\bm{u}}_h^T\bm{v})\cdot \bm{n}_i
\, \mathrm{d}s\Big|\lesssim \lambda_2(\nabla\underline{\bm{u}}_h)
\int_{e_i}|\bm{v}|\,\mathrm{d}s
\cong\dfrac{h_T}{\sigma}\int_{\hat{e}_i}|\hat{\bm{v}}|
\,\mathrm{d}\hat{s}\lesssim \dfrac{h_T}{\sigma}\|\hat{\bm{v}}\|_{1,2,\hat{T}}.
\end{equation}
Similarly, since $\lambda_1(\operatorname{cof}\nabla\underline{\bm{u}}_h)\gtrsim
\sigma$, we have
\begin{equation}\label{denominator}
\begin{aligned}
\Big|\int_{e_i}(\operatorname{cof}\nabla\underline{\bm{u}}_h^T\bm{p}_i)\cdot\bm{n}_i
\,\mathrm{d}s\Big| =\Big|\int_{e_i}\bm{p}_i\cdot(\operatorname{cof}\nabla
\underline{\bm{u}}_h\bm{n}_i) \,\mathrm{d}s\Big|
\gtrsim \sigma h_T\int_{\hat{e}_i}\hat{q}_i \,\mathrm{d}\hat{s}.
\end{aligned}
\end{equation}
Therefore, \eqref{alpha i  definition}-\eqref{denominator} yields that
\begin{equation}
\label{alpha i}
\begin{aligned}
|\alpha_i|\lesssim \dfrac{1}{\sigma^2}\|\hat{\bm{v}}\|_{1,2,\hat{T}},\quad i=1,2,3,4.
\end{aligned}
\end{equation}
Hence, by the standard scaling argument, we obtain
\begin{equation}
\label{bar pi2 L2}
|\tilde{\Pi}_2\bm{v}|^2_{1,2,T}=\Big|\sum_{i=1}^4\alpha_i\bm{p}_i\Big|^2_{1,2,T}\lesssim
\dfrac{1}{\sigma^4}(h_T^{-2}\|\bm{v}\|^2_{0,2,T}+|\bm{v}|^2_{1,2,T}).
\end{equation}
Since $\nabla\bar{\Pi}_2\bm{v}=\nabla\tilde{\Pi}_2\bm{v}$, it follows from
\eqref{Clement2} and \eqref{bar pi2 L2} that
\begin{equation*}
\begin{aligned}
|\Pi_1\bm{v}|^2_{1,2,\Omega }
\lesssim &|\bar{\Pi}_1\bm{v}|^2_{1,2,\Omega }
+\sum_{T}|\bar{\Pi}_2(\bm{v}-\bar{\Pi}_1\bm{v})|^2_{1,2,T}\\ \lesssim
&|\bar{\Pi}_1\bm{v}|^2_{1,2,\Omega } +\sum_{T}\dfrac{1}{\sigma^4}(h_T^{-2}
\|\bm{v}-\bar{\Pi}_1\bm{v}\|^2_{0,2,T}+|\bm{v}-\bar{\Pi}_1\bm{v}|^2_{1,2,T})\lesssim
\dfrac{1}{\sigma^4}|\bm{v}|^2_{1,2,\Omega }.
\end{aligned}
\end{equation*}
Recall $\int_{\Omega } \bm{v}\,\mathrm{d}\bm{x}=0$, $\forall \bm{v}\in \mathcal{X}_0$,
this and the Poincar\'{e} inequality lead to the inequality in \eqref{operator1 property}.
On the other hand, by \eqref{third condition}, we have, for all $\bm{v}\in \mathcal{X}_0$,
\begin{equation*}
\int_T\operatorname{cof}\nabla \underline{\bm{u}}_h:
\nabla(\Pi_1\bm{v}-\bm{v})\,\mathrm{d}\bm{x}=\int_T\operatorname{cof}\nabla \underline{\bm{u}}_h:
\nabla\big(\bar{\Pi}_2
(\bm{v}-\bar{\Pi}_1\bm{v})-(\bm{v}-\bar{\Pi}_1\bm{v})\big)\,\mathrm{d}\bm{x}=0.
\end{equation*}
This completes the proof of the lemma.
\end{proof}

\begin{Lemma}\label{int_det}
Let $\mathscr{T}_h$ and $\underline{\bm{u}}_h$ satisfy the mesh regularity conditions (M1)-(M2)
and the deformation regularity conditions (C1)-(C2) respectively. Let
$\hat{b} = (1-\hat{x}_1^2)(1-\hat{x}_2^2)$ be the bubble function on
$\hat{T}$. Then
\begin{equation}
\det\Big(\int_{\hat{T}}\hat{b} \operatorname{cof}\nabla_{\hat{x}}\underline{\hat{\bm{u}}}_h(\bm{x})
\,\mathrm{d}\bm{\hat{x}}\Big)\cong h_T^2.
\label{det int b cof nabla u =h^2}
\end{equation}
where the gradient operator $\nabla_{\hat{x}}:=(\partial_{\hat{x}_1}, \partial_{\hat{x}_2})$.
\end{Lemma}
\begin{proof}
    Recall $\underline{\bm{u}}_h = (u_1, u_2)$, and $\hat{u}_{ij} = \sum_{k=0}^8 u_i( a_k)\frac{\partial \varphi_k}{\partial\hat{x}_j}$ on $\hat{T}$, $1\le i,j \le 2$, where $\varphi_k$ are bi-quadratic basis functions. Rewrite $\int_{\hat{T}}\hat{b}\operatorname{cof}\nabla_{\hat{x}}\hat{\underline{\bm{u}}}\,\mathrm{d}\hat{\bm{x}}$ as
\begin{equation}
\int_{\hat{T}}\hat{b}\operatorname{cof}\nabla_{\hat{x}}\underline{\bm{\hat{u}}}_h\,\mathrm{d}\hat{\bm{x}} =
\left(
\begin{aligned}
&\int_{\hat{T}}\hat{b}\hat{u}_{22}\,\mathrm{d}\hat{\bm{x}} &-\int_{\hat{T}}\hat{b}\hat{u}_{21}\,\mathrm{d}\hat{\bm{x}}\\
&-\int_{\hat{T}}\hat{b}\hat{u}_{12}\,\mathrm{d}\hat{\bm{x}}& \int_{\hat{T}}\hat{b}\hat{u}_{11}\,\mathrm{d}\hat{\bm{x}}\\
\end{aligned}
\right).
\end{equation}
By Taylor expanding $u_1(a_k)$ at $a_8$, $k = 0,1,\cdots,7$ and direct calculations, we have
\begin{equation}\label{hat_u_11}
\begin{aligned}
\int_{\hat{T}}\hat{b}\hat{u}_{11}\,\mathrm{d}\hat{\bm{x}} &= \dfrac{4}{45}(u_1( a_1)-u_1( a_0)
+ u_1( a_2) - u_1( a_3)) + \dfrac{32}{45}(u_1( a_5)-u_1( a_7))\\
& = \dfrac{\partial u_1}{\partial x_1}(a_8)\left(\dfrac{4}{45}(x_{a_1}-x_{a_0} + x_{a_2}
- x_{a_3})+\dfrac{32}{45}(x_{a_5}-x_{a_7})\right)\\
&\quad + \dfrac{\partial u_1}{\partial x_2}(a_8)\left(\dfrac{4}{45}(y_{a_1}-y_{a_0} + y_{a_2}
- y_{a_3})+\dfrac{32}{45}(y_{a_5}-y_{a_7})\right) + C_{11}h_T^2\\
&= \dfrac{4}{45} \nabla u_1(a_8)\cdot \left(\vv{a_0a_1} + \vv{a_3a_2} + 8 \vv{a_7a_5}\right)
+ C_{11}h_T^2.\\
\end{aligned}
\end{equation}
where $ a_k = (x_{a_k},y_{a_k})$, $k=0,\cdots,8$. Similarly, we have
\begin{equation}\label{hat_u_12}
\begin{aligned}
\int_{\hat{T}}\hat{b}\hat{u}_{12}\,\mathrm{d}\hat{\bm{x}} &= \dfrac{4}{45}(u_1(a_3) - u_1(a_0)
+ u_1(a_2) - u_1(a_1)) +\dfrac{32}{45}(u_1(a_6)-u_1(a_4))\\
&= \dfrac{4}{45}\nabla u_1(a_8)\cdot\left(\vv{a_0a_3} + \vv{a_1a_2} + 8\vv{a_4a_6}\right)
+ C_{12}h_T^2,
\end{aligned}
\end{equation}
\begin{equation}\label{hat_u_21}
\begin{aligned}
\int_{\hat{T}}\hat{b}\hat{u}_{21} \,\mathrm{d}\hat{\bm{x}}& = \dfrac{4}{45}(u_2(a_1)-u_2(a_0)
+ u_2(a_2) - u_2(a_3)) + \dfrac{32}{45}(u_2(a_5)-u_2(a_7))\\
&= \dfrac{4}{45}\nabla u_2(a_8)\cdot \left(\vv{a_0a_1} + \vv{a_3a_2} + 8 \vv{a_7a_5}\right)
+ C_{21}h_T^2,
\end{aligned}
\end{equation}
\begin{equation}\label{hat_u_22}
\begin{aligned}
\int_{\hat{T}}\hat{b}\hat{u}_{22}\,\mathrm{d}\hat{\bm{x}} &= \dfrac{4}{45}(u_2(a_3) - u_2(a_0)
+ u_2(a_2) - u_2(a_1)) +\dfrac{32}{45}(u_2(a_6)-u_2(a_4))\\
&= \dfrac{4}{45}\nabla u_2(a_8)\cdot\left(\vv{a_0a_3} + \vv{a_1a_2} + 8\vv{a_4a_6}\right)
+ C_{22}h_T^2.
\end{aligned}
\end{equation}

Thus, by the mesh and deformation regularity conditions (M1)-(M2) and (C1)-(C2),
and noticing that $C_{ij}$ of the $h_T^2$ terms in \eqref{hat_u_11}-\eqref{hat_u_22} are
linear combinations of $\partial^2 u_i/\partial x_1^2$, $\partial^2 u_i/\partial x_2^2$
and $\partial^2 u_i/\partial x_1x_2$ with uniformly bounded coefficients, we are led to
\begin{equation*}
\begin{aligned}
 \det\int_{\hat{T}}\hat{b}\operatorname{cof}\nabla_{\hat{x}}\hat{\underline{\bm{u}}}_h
 \,\mathrm{d}\hat{\bm{x}} &=
 \int_{\hat{T}}\hat{b}\hat{u}_{11}\,\mathrm{d}\hat{\bm{x}}\int_{\hat{T}}\hat{b}\hat{u}_{22}
 \,\mathrm{d}\hat{\bm{x}} - \int_{\hat{T}}\hat{b}\hat{u}_{21}\,\mathrm{d}\hat{\bm{x}}
 \int_{\hat{T}}\hat{b}\hat{u}_{12}\,\mathrm{d}\hat{\bm{x}}\\
& \cong \Big[\left(\dfrac{4}{45}\nabla u_1(a_8)\cdot \bm{l}_1\right)
\left(\dfrac{4}{45}\nabla u_2(a_8)\cdot \bm{l}_2\right) \\
&\quad -\left(\dfrac{4}{45}\nabla u_1(a_8)\cdot \bm{l}_2\right)
\left(\dfrac{4}{45}\nabla u_2(a_8) \cdot \bm{l}_1\right)\Big] + O(h_T^2)\\
& \cong \det\nabla u_h(a_8) h_T^2 + O(h_T^2)\cong h_T^2,
\end{aligned}
\end{equation*}
where $\bm{l}_1 = (\vv{a_0a_1} +\vv{a_3a_2} + 8\vv{a_7a_5})$ and
$\bm{l}_2 = (\vv{a_0a_3} +\vv{a_1a_2} + 8\vv{a_4a_6})$. This proves \eqref{det int b cof nabla u =h^2}.
\end{proof}

\begin{Theorem} Suppose the hypothesis (H), the conditions (M1)-(M2) and (C1)-(C2) hold.
Then, there exists a constant $\beta>0$ independent of $h$ such that
$b(\bm{v}_h,q_h;\underline{\bm{u}}_h)$ satisfies the LBB condition \eqref{LBB condition}.
\end{Theorem}

\begin{proof}
Set $\Pi_h=\Pi_1+\Pi_2(I-\Pi_1)$. Then, by Lemma~\ref{Fortin Criterion},
Lemma~\ref{2 steps construction} (see \eqref{operator construction}), and
\eqref{operator2 construction}-\eqref{operator1 property}, what remains
for us to show is that $\Vert \Pi_2(I-\Pi_1)\bm{v}\Vert _{1,2,\Omega }\le c_2 \Vert
\bm{v}\Vert _{1,2,\Omega }$, $\forall \bm{v}\in \mathcal{X}_0$.

Since $\Pi_2\bm{v}$ is a bubble function on $T \in \mathscr{T}_h$ and
$\operatorname{div}(\operatorname{cof}\nabla\bm{\underline{u}}_h|_T)=0$,
by the integral by parts and the change of integral variables,
equation \eqref{operator2 construction} can be rewritten as
\begin{equation}
\label{rewrite pi2 hat}
\int_{\hat{T}}\widehat{\Pi_2\bm{v}}\cdot \big(\operatorname{cof}\nabla_{\hat{x}}
\underline{\bm{\hat{u}}}_h\nabla_{\hat{x}}\hat{q}_h\big) \, \mathrm{d}\bm{\hat{x}}
= - \int_{\hat{T}} \hat{q}_h \operatorname{cof}\nabla_{\hat{x}} \underline{\bm{\hat{u}}}_h:
\nabla_{\hat{x}}\bm{\hat{v}} \,\mathrm{d}\bm{\hat{x}}, \;\;
\forall \hat{q}_h\in P_1(\hat{T})\setminus P_0(\hat{T}).
\end{equation}
where $\nabla_{\hat{x}}:=(\partial_{\hat{x}_1}, \partial_{\hat{x}_2})$ is the gradient operator.
By solving the linear system, we can write $\widehat{\Pi_2\bm{v}}(\hat{\bm{x}})$ explicitly as
$\widehat{\Pi_2\bm{v}}(\hat{\bm{x}})= (\alpha_1(1-\hat{x}_1^2)(1-\hat{x}_2^2),
\alpha_2(1-\hat{x}_1^2)(1-\hat{x}_2^2))$ with
\begin{equation}
\label{bm alpha}
\bm{\alpha} = \begin{pmatrix} \alpha_1 \\ \alpha_2 \end{pmatrix} =
- \left(\int_{\hat{T}}\hat{b}\operatorname{cof}\nabla_{\hat{x}}\underline{\hat{\bm{u}}}_h
\,\mathrm{d}\bm{\hat{x}}\right)^{-1} \begin{pmatrix} \int_{\hat{T}}
\operatorname{cof}\nabla_{\hat{x}}\underline{\hat{\bm{u}}}_h: \nabla_{\hat{x}}\hat{\bm{v}}
\ \hat{x}_1\,\mathrm{d}\bm{\hat{x}} \\
\int_{\hat{T}}\operatorname{cof}\nabla_{\hat{x}}\underline{\hat{\bm{u}}}_h:
\nabla_{\hat{x}}\hat{\bm{v}}\ \hat{x}_2\,\mathrm{d}\bm{\hat{x}},
\end{pmatrix}
\end{equation}
where $\hat{b}(\bm{\hat{x}})=(1-\hat{x}_1^2)(1-\hat{x}_2^2)$.
By (C1) and the H\"{o}lder inequality,
\begin{equation}
\Big|\int_{\hat{T}}\operatorname{cof}\nabla_{\hat{x}}\underline{\hat{\bm{u}}}_h:
\nabla_{\hat{x}}\bm{\hat{v}}\ \hat{x}_i\,\mathrm{d}\bm{\hat{x}}\Big|
\lesssim \dfrac{h_T}{\sigma}|\hat{\bm{v}}|_{1,2,\hat{T}}\|\hat{x}_i\|_{0,2,T}
\lesssim \dfrac{h_T}{\sigma}|\hat{\bm{v}}|_{1,2,T},\quad i=1,2.
\end{equation}
On the other hand, by (C1) and Lemma~\ref{int_det},
\begin{equation}
\label{coefficient matrix}
\Big|\big(\int_{\hat{T}}\hat{b}\operatorname{cof}\nabla_{\hat{x}}\underline{\hat{\bm{u}}}_h
\, \mathrm{d}\bm{\hat{x}}\big)^{-1}\Big|\cong h_T^{-2}\Big|\int_{\hat{T}}\hat{b}
\nabla_{\hat{x}}\underline{\hat{\bm{u}}}_h\,\mathrm{d}\bm{\hat{x}}\Big|
\lesssim h_T^{-2} \|\hat{b}\|_{0,2,\hat{T}}\|\nabla_{\hat{x}}\underline{\hat{\bm{u}}}_h
\|_{0,2,\hat{T}}\lesssim \dfrac{1}{\sigma h_T}.
\end{equation}
As a consequence of \eqref{bm alpha}-\eqref{coefficient matrix} and the standard scaling argument,
we are led to
\begin{equation}
\label{pi2 bubble}
|\Pi_2\bm{v}|_{1,2,T}\cong |\widehat{\Pi}_2\bm{v}|_{1,2,\hat{T}}\cong |\bm{\alpha}|\lesssim
\dfrac{1}{\sigma^2}|\hat{\bm{v}}|_{1,2,\hat{T}}\cong \dfrac{1}{\sigma^2}|\bm{v}|_{1,2,T}.
\end{equation}
Finally, by \eqref{operator1 property}, \eqref{pi2 bubble} and
the Poincar\'{e} inequality, we obtain
\begin{equation}
\label{beta magnitude}
\|\Pi_2(I-\Pi_1)\bm{v}\|_{1,2,\Omega }\lesssim \dfrac{1}{\sigma^2}\|(I-\Pi_1)\bm{v}
\|_{1,2,\Omega } \lesssim\dfrac{1}{\sigma^4}\|\bm{v}\|_{1,2,\Omega },
\  \forall \bm{v}\in \mathcal{X}_0,
\end{equation}
and complete the proof of the theorem.
\end{proof}

\section{Numerical experiments and results}

In this section, we apply a specific DP-Q2-P1 method to a typical cavitation problem in
incompressible nonlinear elasticity. As is well known that cavitation refers to a commonly
observed phenomenon in elastomers, in which small voids enlarge by one or more orders of
magnitude when subject to sufficiently large tensile stresses. The numerical computation
of cavitation is difficult because the extremely large anisotropic deformation near the
cavity surface in the form of increasingly severe compression in the radial direction
and correspondingly large stretches in circumferential one, which can often cause mesh tangle
and other approximation problems \cite{SuLiRectan, Lian Dual, Xu and Henao 2011}.

Let $\Omega = B_1(\bm{0})\setminus B_{\rho}(\bm{0})$ be the reference configuration,
where $\rho$ is the radius of the pre-existing defect, let $\partial_N\Omega = \partial\Omega$.
Let the strain-energy density
function be given by
$$
W(\nabla\bm{u}) = \frac{\mu}{2}|\nabla\bm{u}|^s + \dfrac{1}{2}(\det\nabla\bm{u}-1)^2 +
\dfrac{1}{\det\nabla\bm{u}},\qquad 1<s<2.
$$

Consider the transformation $F_T:\hat{T}\to \mathbb{R}^2$ of the form
\begin{align}\label{F_T}
\begin{cases}
R=R_0+\dfrac{\hat{x}_1 +1}{2}(R_1-R_0), \\
\theta=\theta_0+\dfrac{\hat{x}_2 +1}{2}(\theta_3-\theta_0), \\
x_1=R\cos\theta, \ x_2=R\sin\theta.
\end{cases}
\end{align}
Then a typical mesh $\mathscr{T}_h$ consisting of well defined circular ring sector elements
on $\Omega = B_1(\bm{0})\setminus B_\rho(\bm{0})$ is shown in Figure~\ref{typical mesh},
where we have $N=8$ evenly spaced elements in each of the 3 circular ring layers. A typical
circular ring sector element $T$, in a prescribed circular ring with inner radius $\rho_T$ and
thickness $\tau_T$, is shown in Figure~\ref{typical element}.
\begin{figure}[H]
\hspace*{-2.5mm}
	\begin{minipage}[l]{0.5\textwidth}
    \centering    \includegraphics[width=2.4in]{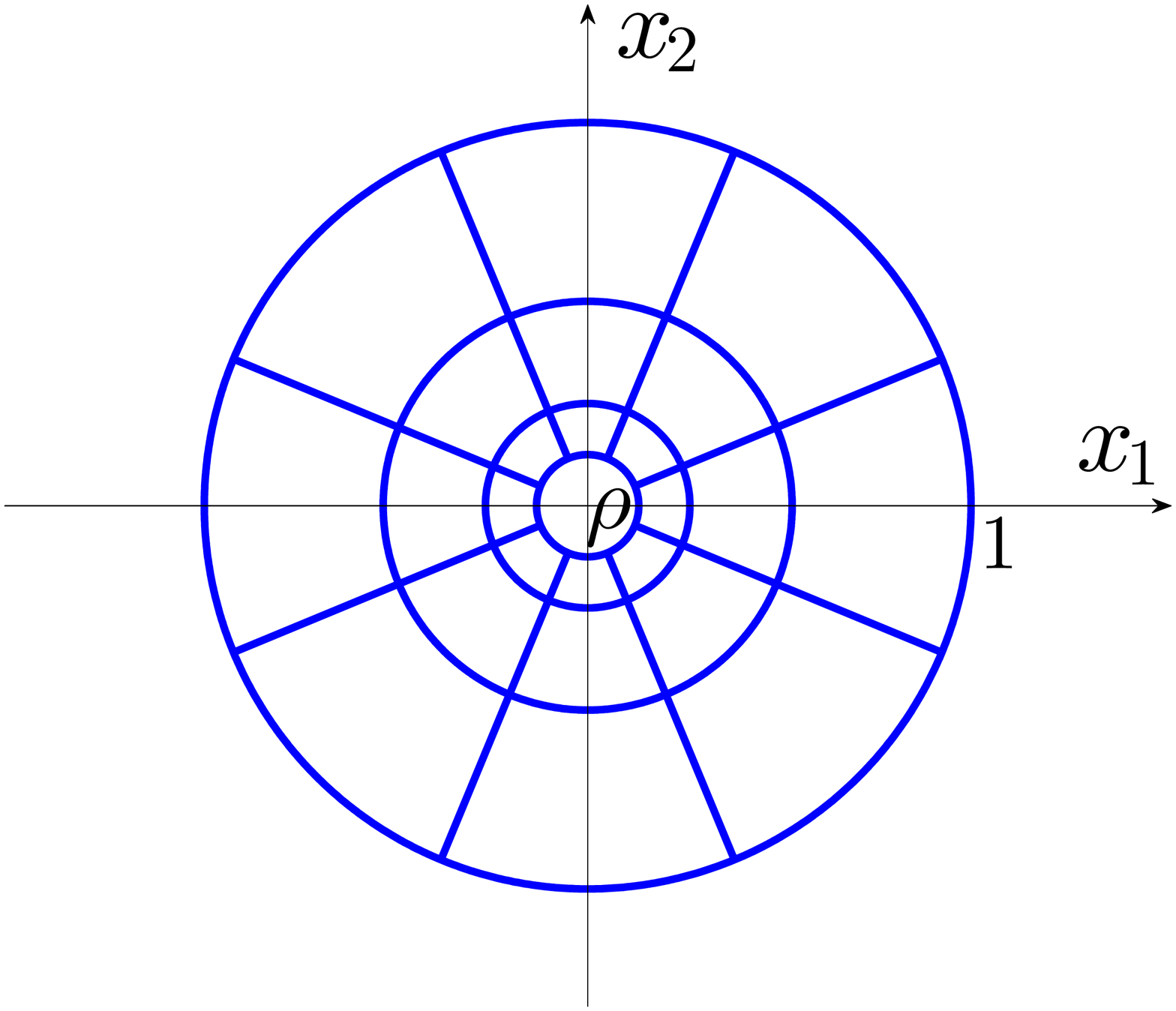}
	\vspace*{-10mm}
	\caption{A typical mesh $\mathscr{T}_h$ with $N=8$.}\label{typical mesh}
    \end{minipage}\hspace*{-.5mm}
    \begin{minipage}[l]{0.5\textwidth}
    \centering   \includegraphics[width=2.4in]{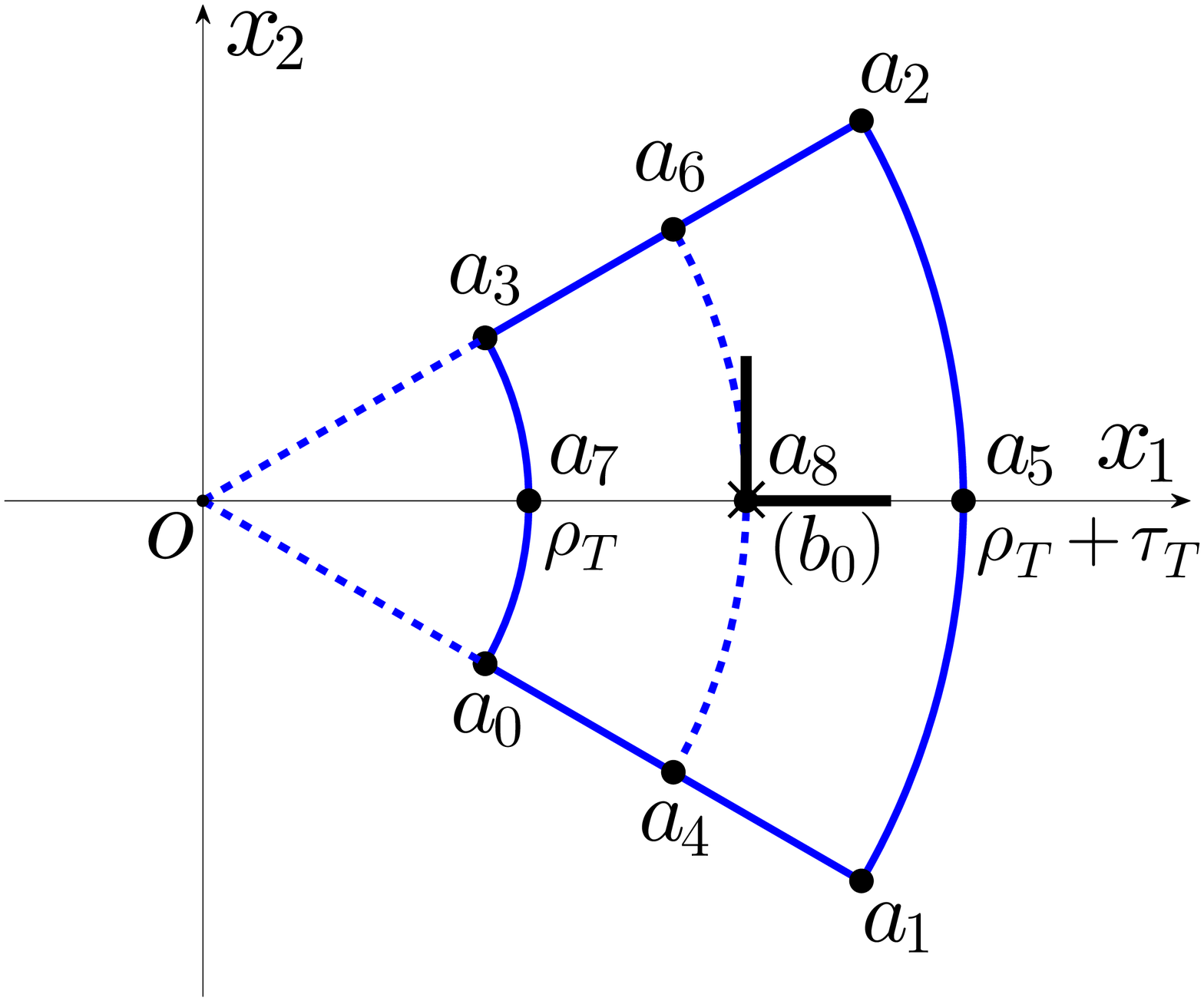}
    \vspace*{-10mm}
    \caption{A circular ring sector element $T$.}\label{typical element}
    \end{minipage}
\end{figure}

In our numerical experiments, the number of elements $N$ in the circular ring layers and
the thickness $\tau_T$ of each layer are determined by a meshing strategy based on an energy
equi-distribution principle established in \cite{SuLiRectan}. For given $\rho >0$, the meshes
so produced satisfy the mesh regularity conditions (M1) and (M2).
Table 1 shows two sets of typical meshes produced by the meshing strategy.
For the constant $\sigma$ in (C1), we set $\sigma = \rho/C_{max}$, where
$C_{max}$ is an upper bound for the expected grown cavity radius.
In our numerical experiments, we set $\sigma = \rho/2$, {\em i.e.} $C_{max}=2$.

\begin{table}[H]{\footnotesize
\centering \subtable[$\rho=0.01$.]{
\begin{tabular}{|c|c|c|c|c|}
\hline
$h$     & $\min\tau_T$  &  $\max\tau_T$  &  layers  & $N$    \\      \hline
0.05    & 0.0300        &  0.1900        &   8      &  20    \\      \hline
0.04    & 0.0224        &  0.1376        &   11     &  26    \\      \hline
0.03    & 0.0156        &  0.1164        &   14     &  34    \\      \hline
0.02    & 0.0096        &  0.0736        &   22     &  50    \\      \hline
\end{tabular}
}
\qquad
\subtable[$\rho=0.0001$.]{
\begin{tabular}{|c|c|c|c|c|}
\hline
$h$     & $\min\tau_T$  &  $\max\tau_T$  & layers  & $N$    \\      \hline
0.05    & 0.0120        &  0.1720        &   9     &  24    \\      \hline
0.04    & 0.0080        &  0.1360        &   12    &  28    \\      \hline
0.03    & 0.0048        &  0.1056        &   16    &  38    \\      \hline
0.02    & 0.0024        &  0.0728        &   22    &  56    \\      \hline
\end{tabular}
}\caption{Data of two sets of typical meshes produced.}\label{table1}
}
\end{table}

\subsection{Radially symmetric case}

In our numerical experiments, we take $\bm{u}(\bm{x};\lambda)=\frac{\sqrt{R^2+\lambda^2-1}}{R}\bm{x}$ with
$\lambda>1$ as the analytical cavitation solution to the radially symmetric
dead-load traction problem with
\begin{equation}\label{symmetric traction}
\bm{t}(\bm{x})=t\bm{n}(\bm{x}), \;\;  \forall \bm{x}\in \partial B_1(\bm{0}), \quad\; \text{and}
\quad\; \bm{t}(\bm{x})=\bm{0}, \;\;  \forall \bm{x}\in \partial B_{\rho}(\bm{0}),
\end{equation}
where $\bm{n}$ is the unit outward normal to $\partial B_1(\bm{0})$, and
$t=t_{\rho,\lambda}$ is uniquely determined by $\rho$ and $\lambda$ \cite{Ball82}.
For example $t_{0.1,2}\approx 3.00487$, $t_{0.01,2}\approx 3.94237$, $t_{0.0001,2}\approx 4.21590$.

\begin{figure}[htb]
\centering \subfigure[Energy error $\Delta E=|E(\bm{u}_h)-E(\bm{u})|$.]{
\label{Energy error}
    \includegraphics[width=2.7in, height=2in]{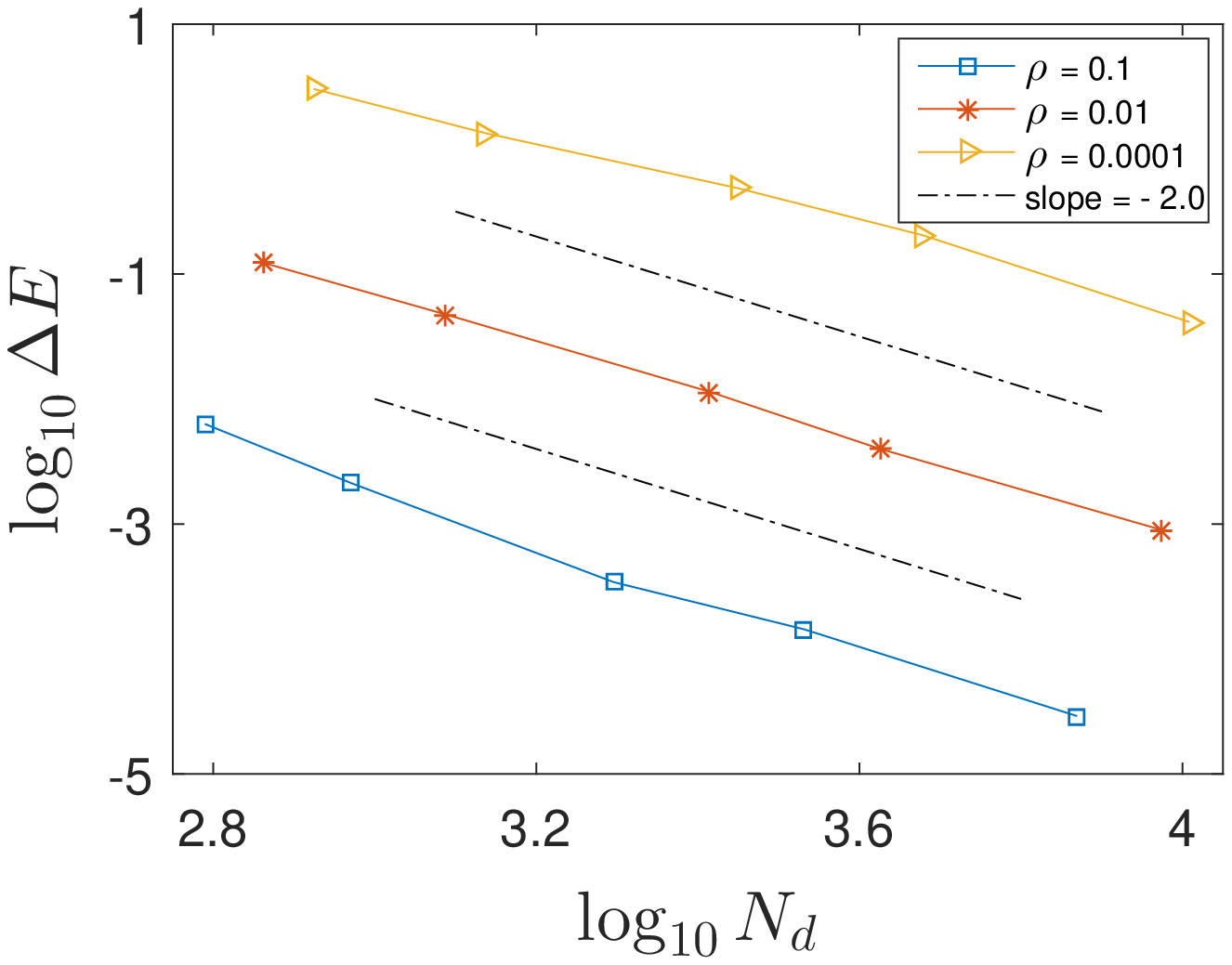}
} \hspace{1mm}
\centering \subfigure[Error in $W^{1,s}$-seminorm $|\bm{u}_h-\bm{u}|_{1,s,\Omega }$.]{
\label{W1p error}
    \includegraphics[width=2.75in, height=2in]{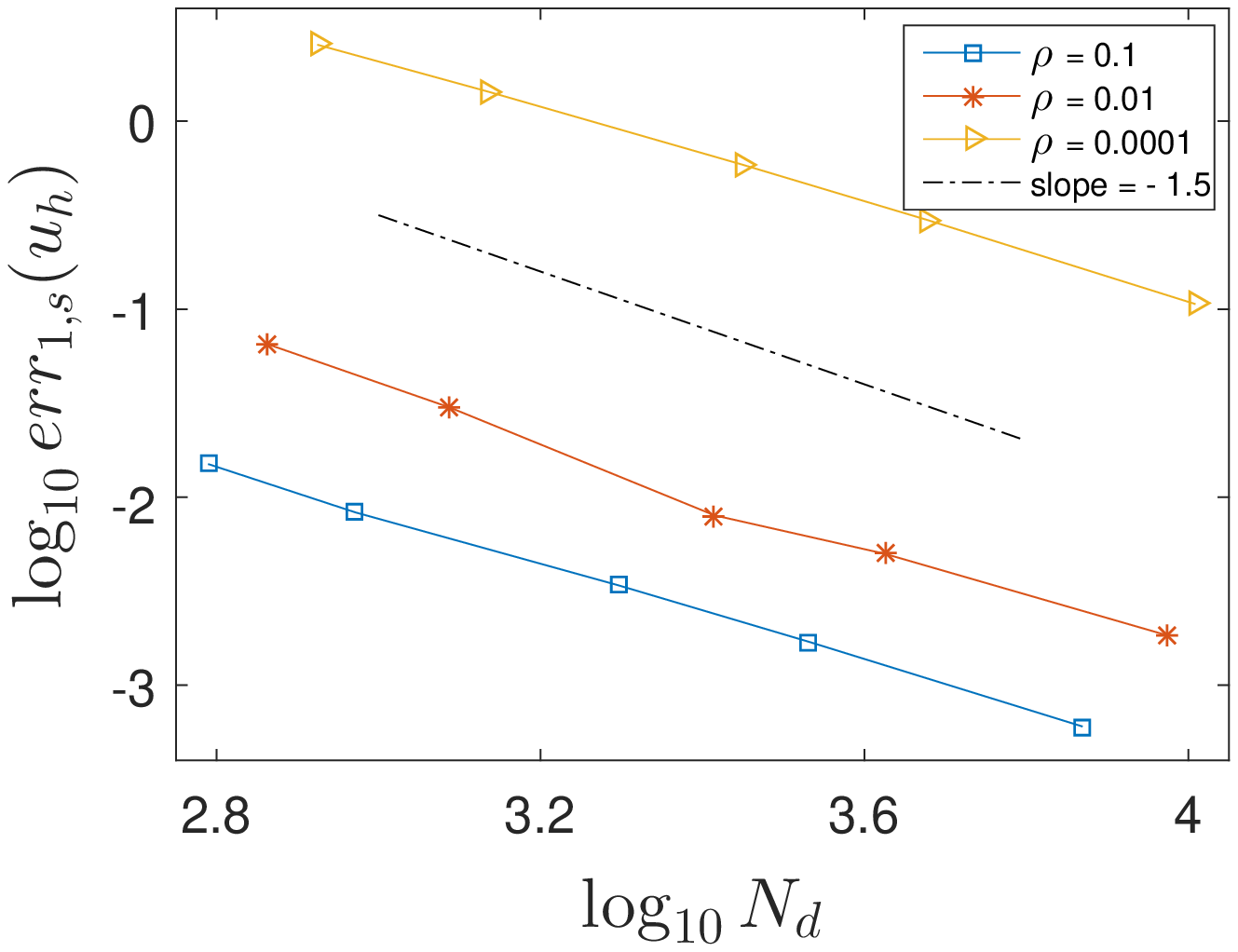}
}
  \caption{Convergence behavior of the energy and deformation in symmetric case.}
\label{Energy error and W1p error}
\end{figure}

The convergence behavior of the numerical cavitation solutions with $\lambda=2$ obtained
by the DP-Q2-P1 mixed finite element method is shown in
Fig~\ref{Energy error and W1p error}-Fig~\ref{p L2 error radial},
where $N_d$ is the total degrees of freedom of $\bm{u}_h$.
Fig~\ref{sym N_d - h} shows $N_d$ as a function
of the mesh size $h$ in the radially symmetric case.
It is clearly seen that $N_d \sim h^{-2}$ for our mesh and the convergence rates
obtained by the DP-Q2-P1 cavitation solutions in the radially symmetric case
can reach the optimal rates of the interpolation error estimates,
which were analyzed in \cite{SuLiRectan} (see Theorem 5.2 in \cite{SuLiRectan}).

\begin{figure}[htb]
\centering \subfigure[$L^1$ error of $\det \nabla u_h$.]{ \label{determinant L1 error}
    \includegraphics[width=2.7in, height=2in]{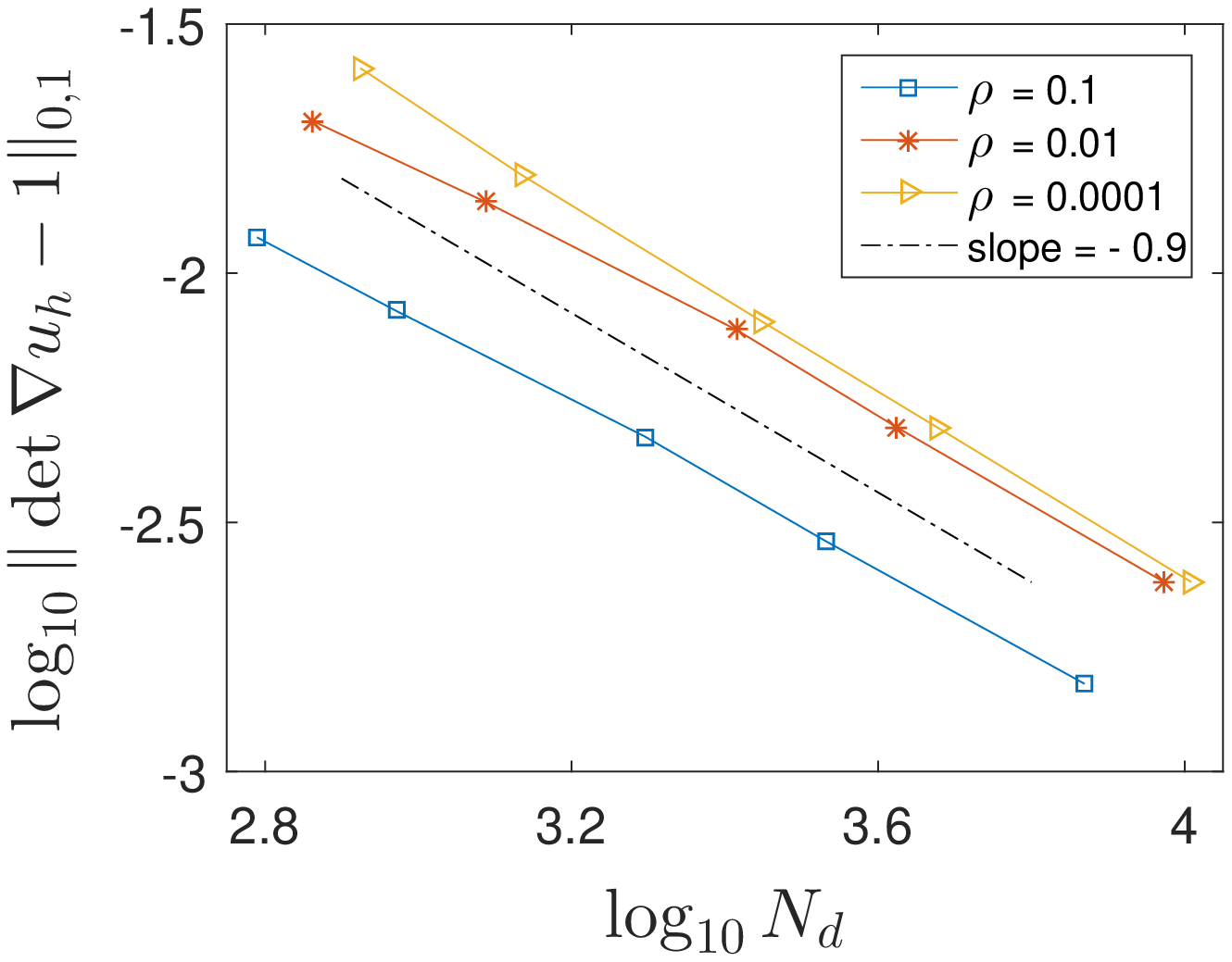}
} \hspace{1mm} \centering \subfigure[$L^2$ error of $\det \nabla u_h$.]{
\label{determinant L2 error}
    \includegraphics[width=2.7in, height=2in]{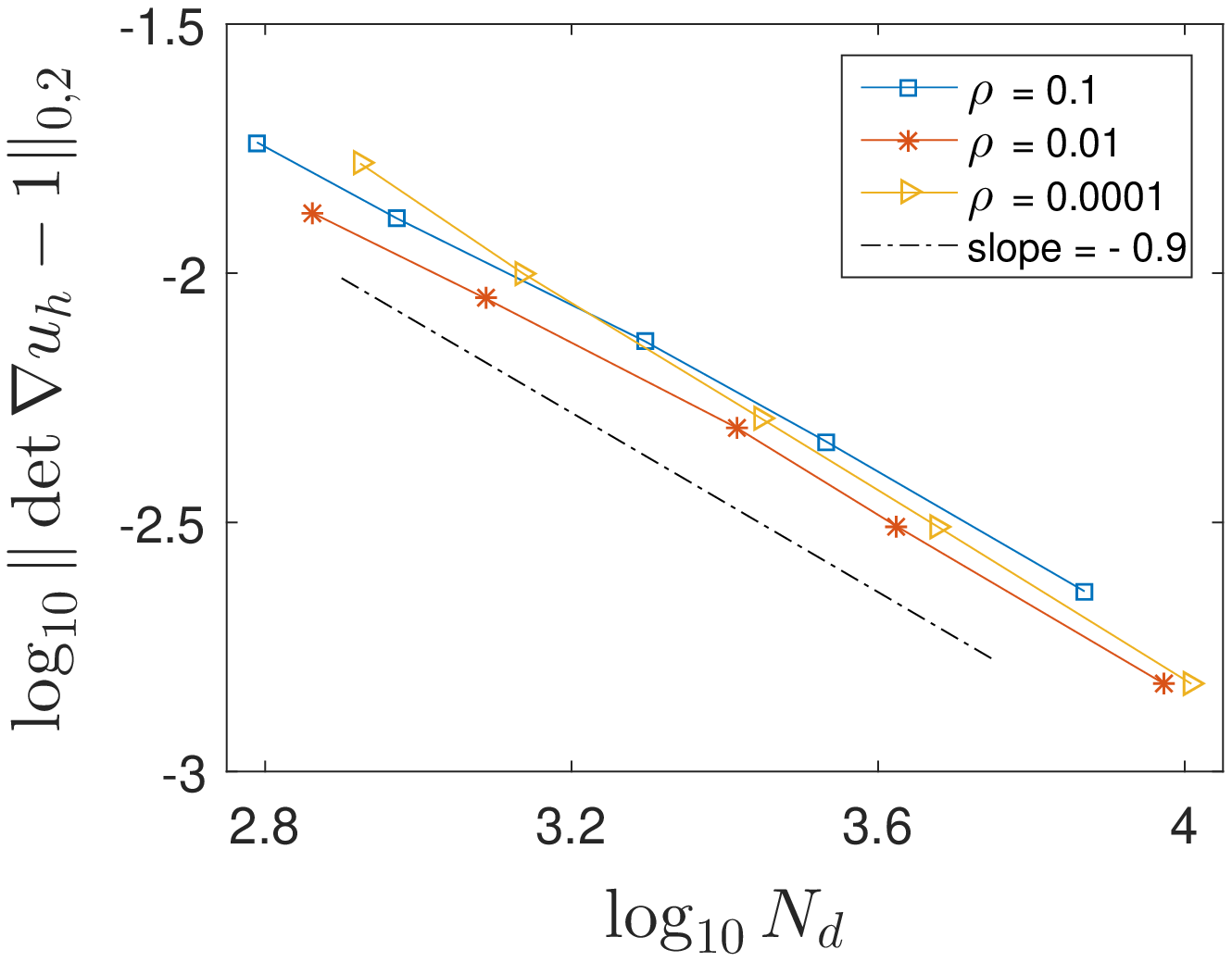}
} \caption{Convergence behavior of $\det \nabla u_h$ in symmetric case.}
\label{convergence det}
\end{figure}

\begin{figure}[htb]
  \begin{minipage}[l]{0.5\textwidth}
\centering
\includegraphics[width=2.7in, height=2in]{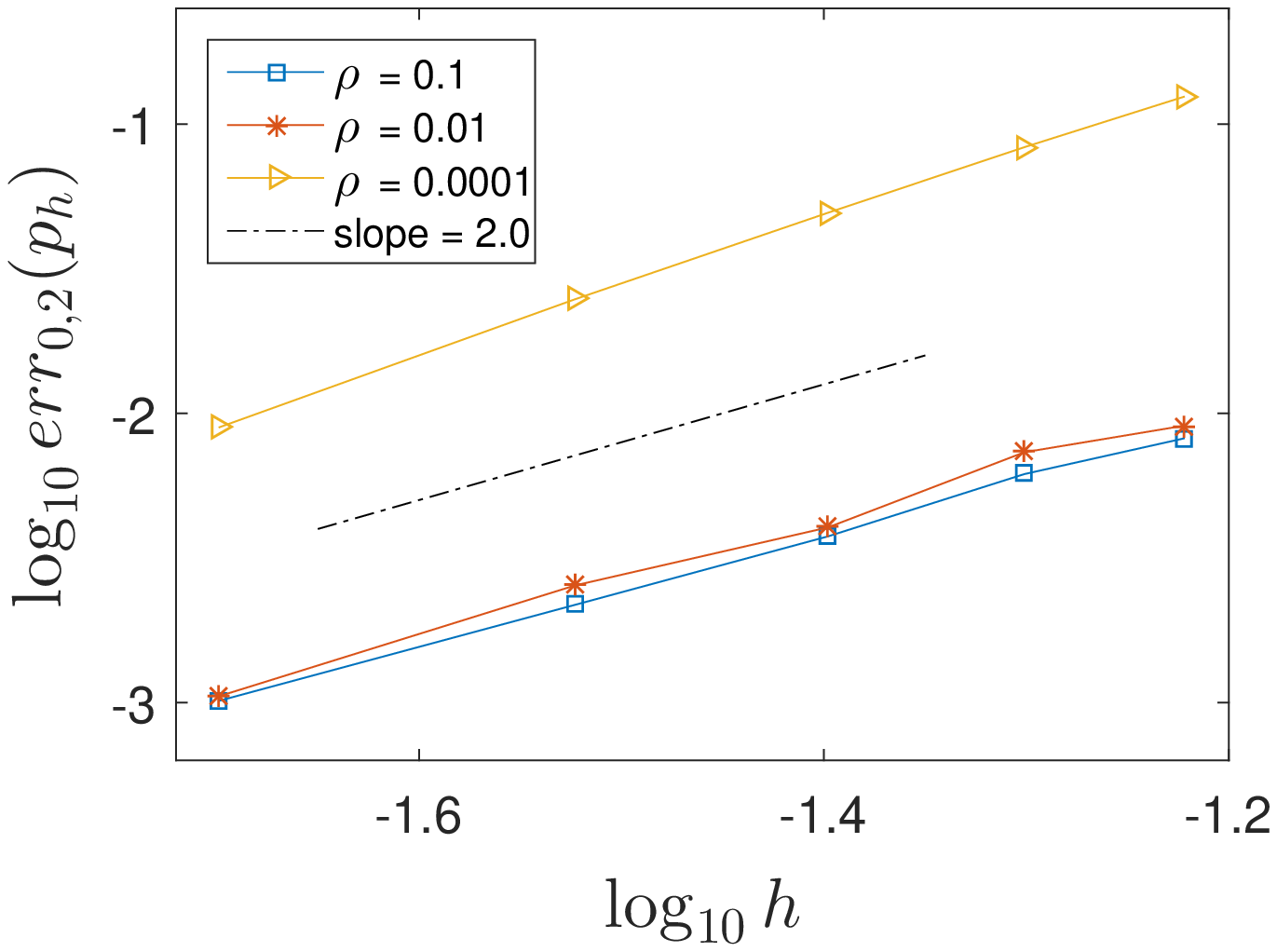}
\vspace*{-3mm}
\caption{Convergence behavior of $p_h$.}
\label{p L2 error radial}
\end{minipage}
  \begin{minipage}[l]{0.5\textwidth}
\centering
\includegraphics[width=2.7in, height=2in]{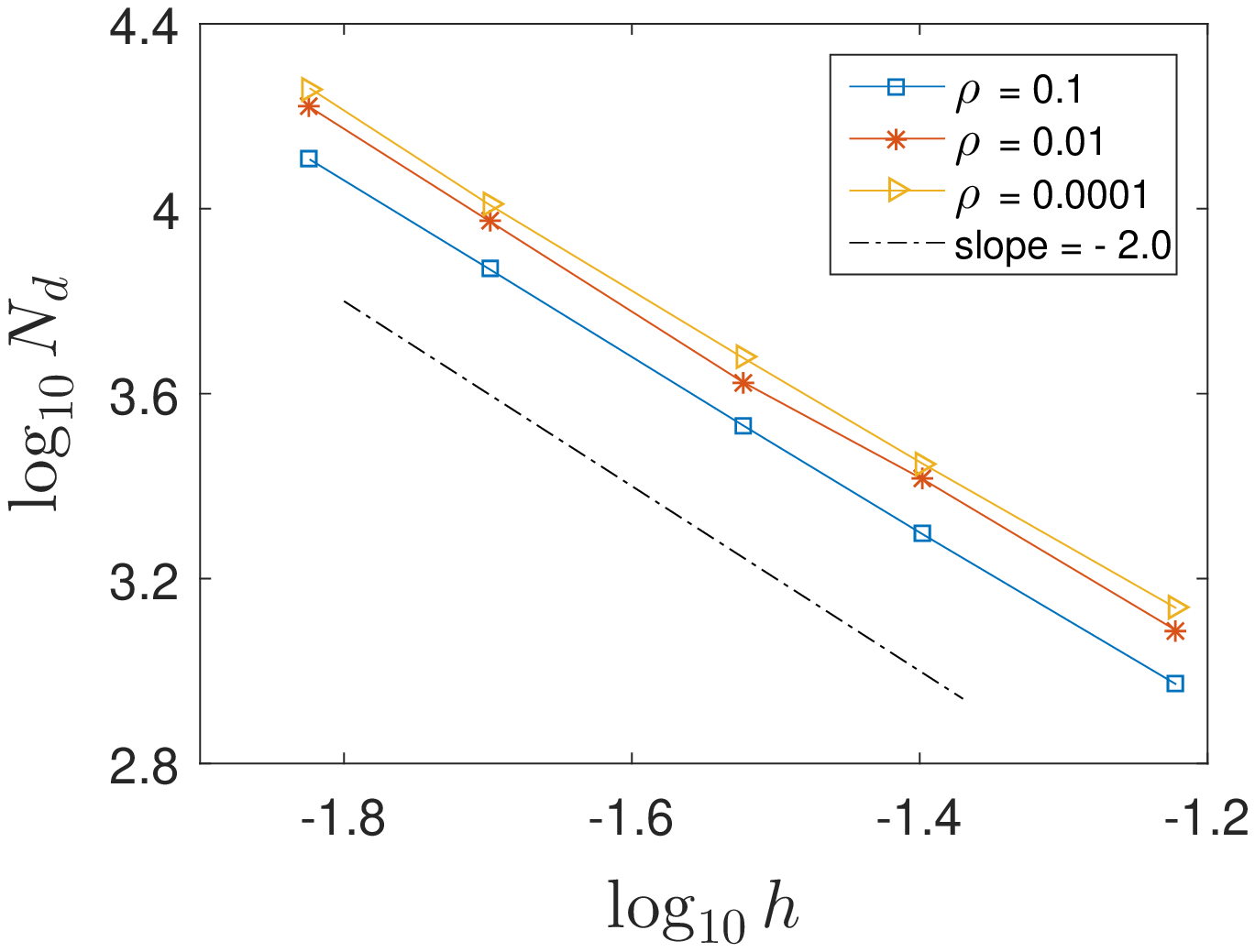}
\vspace*{-3mm}
\caption{$N_d \sim h^{-2}$  in symmetric case.}
\label{sym N_d - h}
\end{minipage}
\end{figure}

\subsection{Non-radially symmetric case}

Consider the non-radially-symmetric dead-load traction problem with
\begin{equation}\label{asymmetric traction}
\bm{t}(\bm{x})=(1+\eta|\cos\theta|) t \bm{n}(\bm{x}), \;\;  \forall \bm{x}\in
\partial B_1(\bm{0}), \quad\; \text{and}
\quad\; \bm{t}(\bm{x})=\bm{0}, \;\;  \forall \bm{x}\in \partial B_{\rho}(\bm{0}),
\end{equation}
where $\theta = \arctan(x_2/x_1)$, $\eta$ and $t$ are parameters.
In our numerical experiments, we take $\eta=1/10$, and
$t=t_{\rho,2}$ as is given in the radially-symmetric
case for various $\rho$.

\begin{figure}[htb]
\centering \subfigure[Energy error $\Delta E=|E(\bm{u}_h)-E(\bm{u})|$.]{
\label{Energy vs Nd nonradial}
    \includegraphics[width=2.7in, height=2in]{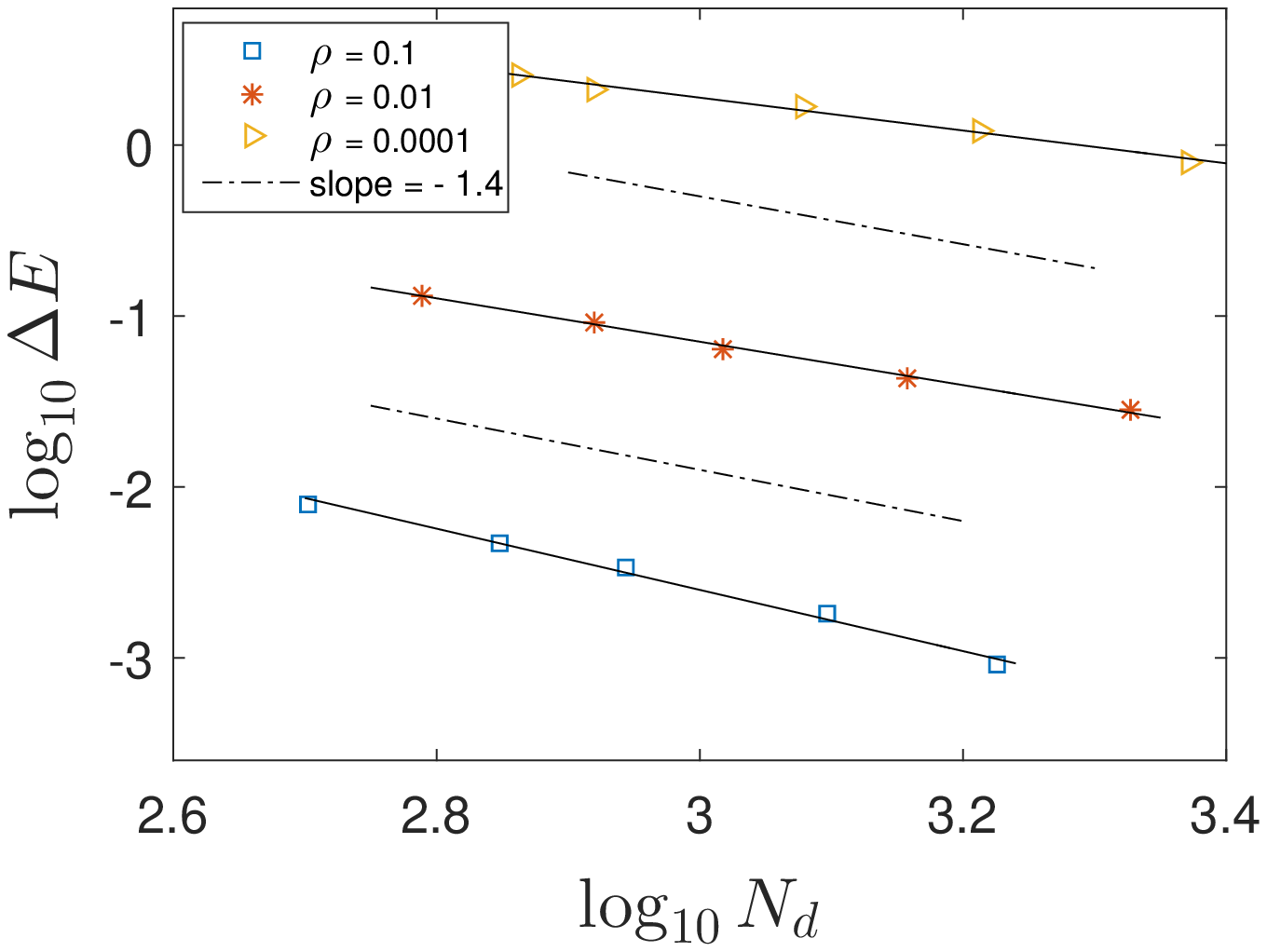}
} \hspace{1mm}
\centering \subfigure[Error in $W^{1,s}$-seminorm $|\bm{u}_h-\bm{u}|_{1,s,\Omega }$.]{
\label{FES error nonradial}
    \includegraphics[width=2.75in, height=2in]{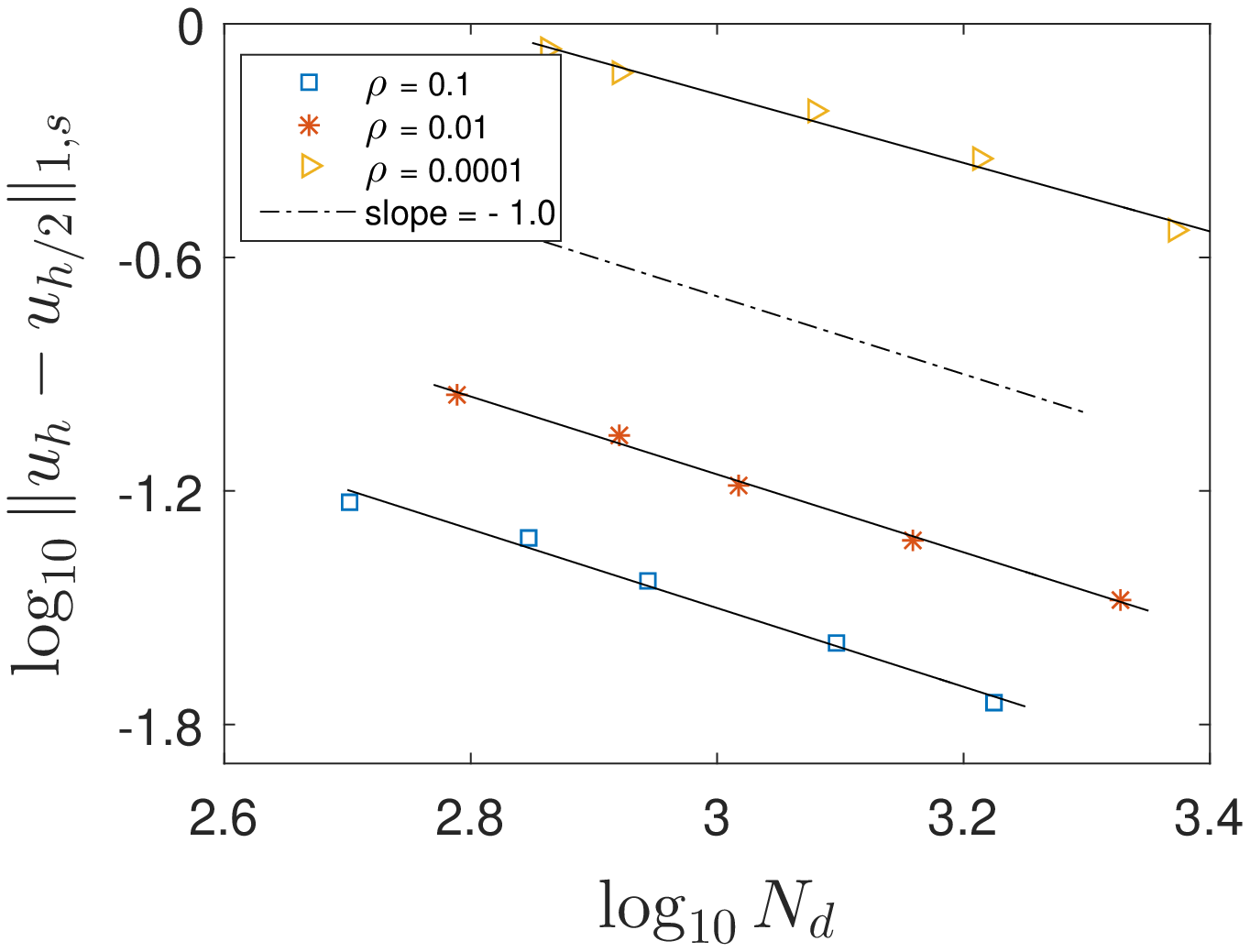}
}
\vspace*{-1mm}
  \caption{Convergence behavior of the energy and deformation in non-symmetric case.}
\label{Energy error and W1p error nonradial}
\end{figure}
\begin{figure}[htb]
\centering \subfigure[$L^1$ error of $\det \nabla u_h$.]{ \label{determinant L1 error nonradial}
    \includegraphics[width=2.7in, height=2in]{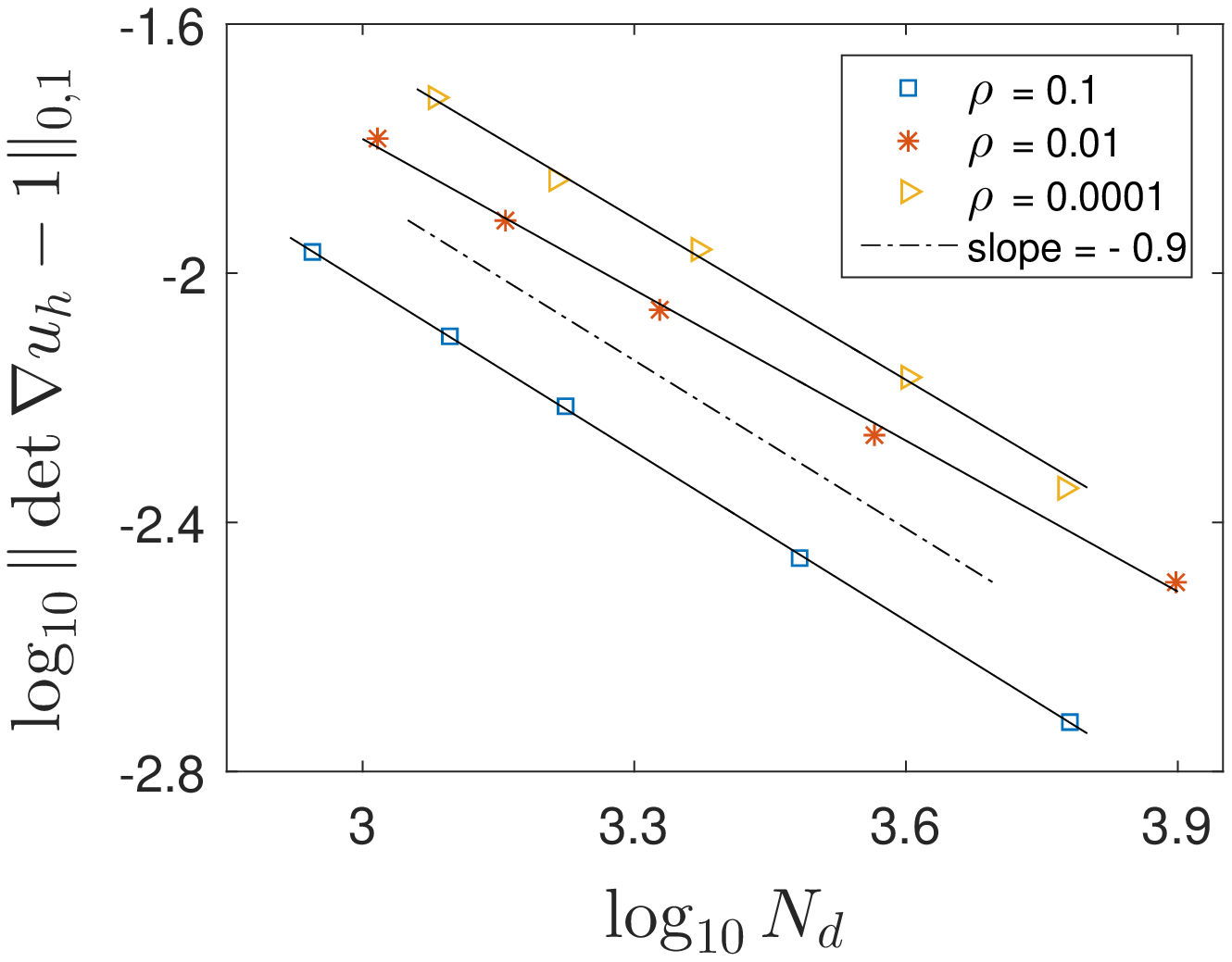}
} \hspace{1mm} \centering \subfigure[$L^2$ error of $\det \nabla u_h$.]{
\label{determinant L2 error nonradial}
    \includegraphics[width=2.7in, height=2in]{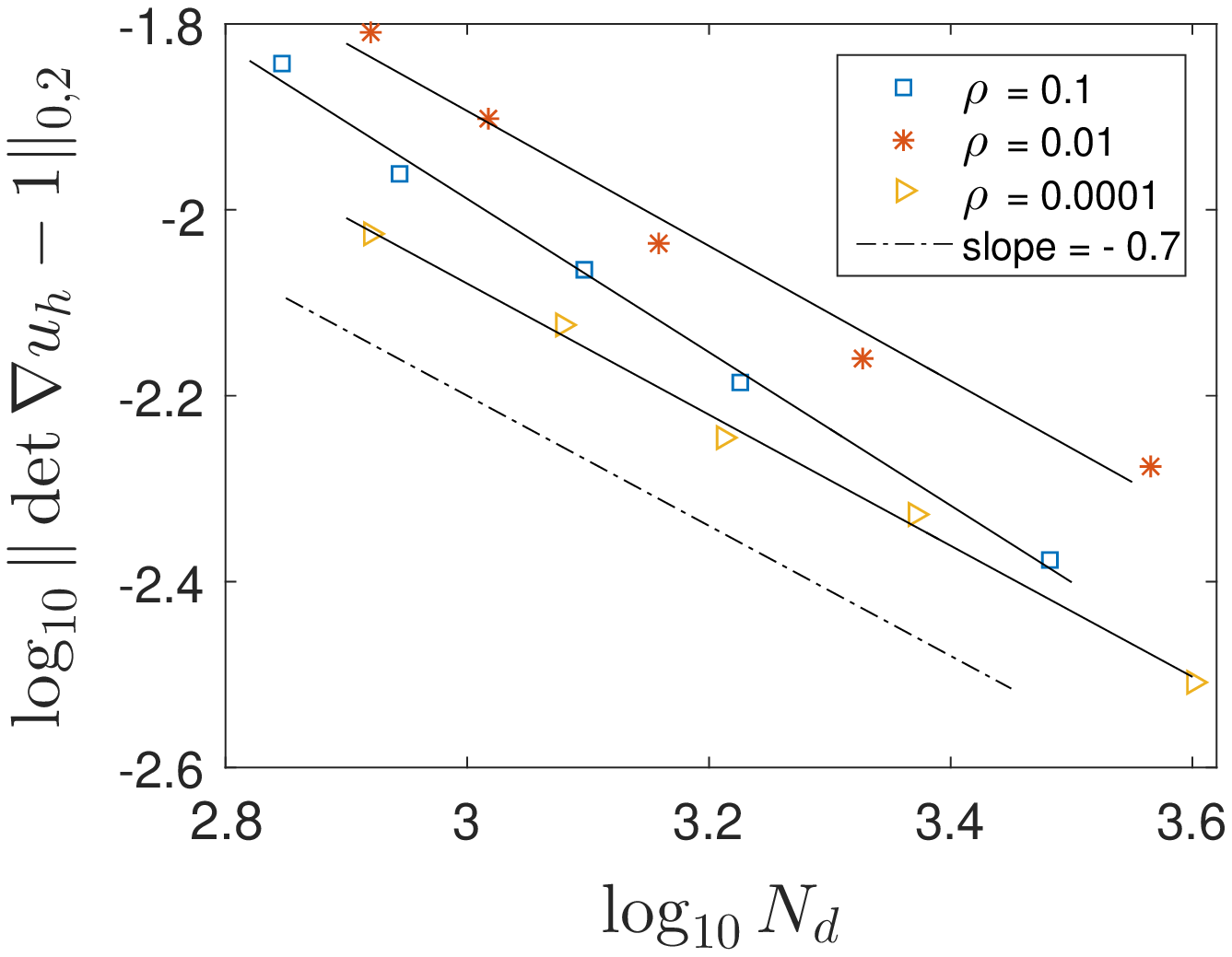}
}
\vspace*{-1mm}
\caption{Convergence behavior of $\det \nabla u_h$ in non-symmetric case.}
\label{convergence det nonradial}
\end{figure}
\begin{figure}[htb!]
  \begin{minipage}[l]{0.495\textwidth}
\centering
\includegraphics[width=2.7in, height=2in]{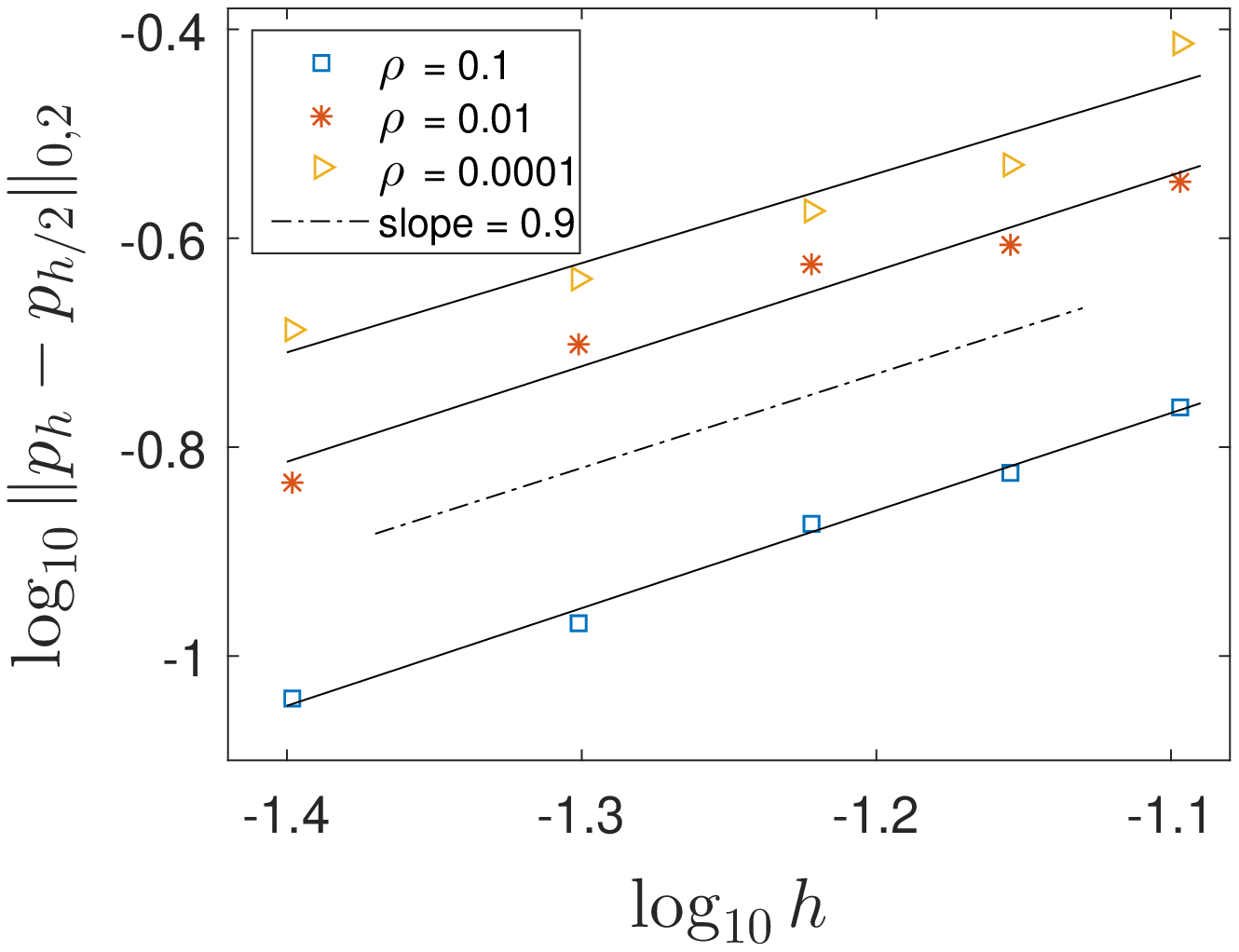}
\vspace*{-3mm}
\caption{Convergence behavior of $p_h$.}
\label{p L2 error nonradial}
\end{minipage}
  \begin{minipage}[l]{0.495\textwidth}
\centering
\includegraphics[width=2.7in, height=2in]{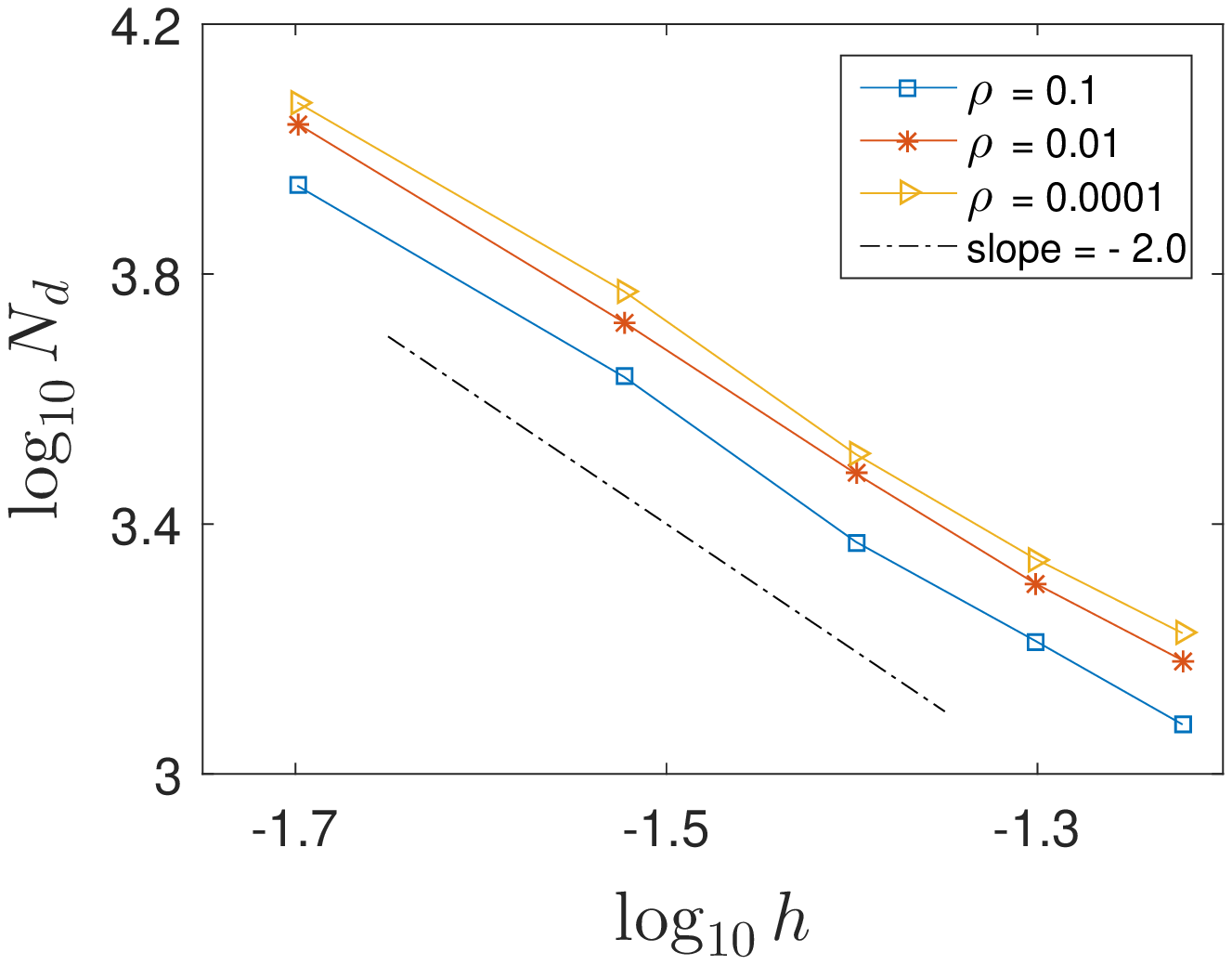}
\vspace*{-3mm}
\caption{$N_d \sim h^{-2}$  in non-symmetric case.}
\label{nonradial sym N_d - h}
\end{minipage}
\end{figure}

The convergence behavior of the numerical cavitation solutions obtained
by the DP-Q2-P1 mixed finite element method is shown in
Fig~\ref{Energy error and W1p error nonradial}-Fig~\ref{p L2 error nonradial}.
Fig~\ref{nonradial sym N_d - h} shows $N_d$ as a function
of $h$ in the non-radially-symmetric case.
We see that, in the non-radially-symmetric case, again $N_d \sim h^{-2}$ for
the meshes produced by the meshing strategy, and the convergence rates obtained
by the DP-Q2-P1 cavitation solutions, though dropped a little bit, are still
close to the optimal rates (see \cite{SuLiRectan}).

\section{Conclusion}

A DP-Q2-P1 mixed finite element method combined with a damped Newton iteration
scheme is established in this paper to numerically solve large deformation problems
in incompressible nonlinear elasticity.
The method is analytically proved to be locking-free and stable.
The numerical experiments on some typical
cavitation problems demonstrate the accuracy and efficiency of the method
in solving incompressible nonlinear elasticity problems with
extremely large anisotropic deformation gradients.

\bibliographystyle{plain}

\setlength{\bibsep}{1ex}

\end{document}